\newcommand{\ba}{\begin{array}}
\newcommand{\ea}{\end{array}}
\newcommand{\bc}{\begin{center}}
\newcommand{\ec}{\end{center}}

\newcommand{\beqn}[1]{\begin{equation}\label{#1}}
\newcommand{\eeqn}{\end{equation}}
\newcommand{\be}{\begin{equation}}
\newcommand{\ee}{\end{equation}}

\newcommand{\beqnn}{\begin{eqnarray}}
\newcommand{\eeqnn}{\end{eqnarray}}

\newtheorem{theorem}{Theorem}
\newtheorem{corollary}{Corollary}
\newtheorem{lemma}{Lemma}
\newtheorem{definition}{Definition}

\newtheorem{assumption}{Assumption}
\newtheorem{problem}{Problem}

\documentclass[12pt,draftcls,onecolumn]{IEEEtran}

\ifCLASSINFOpdf
\else
\fi
\usepackage{graphicx}
\usepackage{cite}
\usepackage{remark}
\usepackage{amssymb}  
\newremark{remark}{Remark}
\usepackage[cmex10]{amsmath}

\usepackage{algorithm}
\usepackage{algpseudocode}
\usepackage{caption}
\usepackage{subcaption}

\begin{document}
%
\title{Receding Horizon Consensus of General Linear Multi-agent Systems with Input Constraints: An Inverse Optimality Approach}

%
%
%

\author{Huiping~Li,~\IEEEmembership{Member,~IEEE,}
        Weisheng~Yan,
        Yang~Shi,~\IEEEmembership{Senior member,~IEEE,}
        Fuqiang~Liu
\thanks{
This paper was not presented at any IFAC
meeting. This work was supported by National Natural Science Foundation of China (NSFC) under Grant 61473225;
the Basic Research Foundation of Northwestern Polytechnical University (NPU) under Grant 3102014JCQ01015; the start-up research fund of NPU.}
\thanks{Huiping Li and Weisheng Yan are with the School of Marine Science and Technology,
Northwestern Polytechnical University, Xi'an, China, 710072 (lihuiping@nwpu.edu.cn;peter.huiping@gmail.com);
Yang Shi is with the Department of Mechanical Engineering, University of Victoria, Victoria,
BC, V8W 3P6 Canada (e-mail: yshi@uvic.ca);
Fuqiang Liu is with the Department of Mechanical Engineering, Chongqing University, Chongqing, China (fuqiangliu@yeah.net).}
}

%


\maketitle

\begin{abstract}
It is desirable but challenging to fulfill system constraints and reach optimal performance
in consensus protocol design for practical multi-agent systems (MASs).
This paper investigates the optimal consensus problem for general linear MASs subject to control input constraints.
Two classes of MASs including subsystems with semi-stable and unstable dynamics are considered.
For both classes of MASs without input constraints, the results on designing optimal consensus protocols are first
developed by inverse optimality approach.
Utilizing the optimal consensus protocols, the receding horizon control (RHC)-based consensus strategies are designed
for these two classes of MASs with input constraints. The conditions for assigning the cost functions distributively are derived, based on which the distributed RHC-based consensus frameworks are formulated.
Next, the feasibility and consensus properties of the closed-loop systems are analyzed.
It is shown that 1) the optimal performance indices under the inverse optimal consensus protocols are coupled with the network topologies and the
system matrices of subsystems, but they are different for MASs with semi-stable and unstable subsystems;
2) the unstable modes of subsystems impose more stringent requirements for the parameter design;
3) the designed RHC-based consensus strategies can make the control input constraints fulfilled and ensure consensus for the closed-loop systems in both cases. But for MASs with semi-stable subsystems, the {\em convergent consensus} can be reached.
Finally, two examples are provided to verify the effectiveness of the proposed results.
\end{abstract}


%
\IEEEpeerreviewmaketitle

\section{INTRODUCTION}\label{sec_introduction}
The consensus problem is one of the most important issues in
multi-agent systems (MASs). It finds many applications in, such as multi-robotic systems,
sensor networks, and power grids, and is also essential to solve some other problems such as formation control,
swarm, and distributed estimation problems.
Many celebrated results have been contributed in the literature of MASs to form the theoretical foundation of consensus problem,
for example, \cite{Reza_04_consensus_TAC}\cite{Moreau_05_TAC_consensus}\cite{Wei_05_TAC_consensus}, just name a few.
Even though much progress has been made in MASs, many practical issues in consensus protocol design
are still left to be explored.

The optimality is a practical requirement in many control systems,
and it is also a desired property for consensus protocol design in MASs. For instance,
a wireless sensor network may be expected to reach consensus in state estimates using
smallest energy as each sensor node has limited battery power.
In addition, the optimal consensus protocol may provide some satisfactory control performance as in LQR.
Another frequently encountered issue would be the control input constraints in MASs.
For example, in a multi-robot system, the control inputs for motors in each robot are not allowed to be too large in order not to ruin the motors, or
the motors may not provide enough power to generate very large control inputs.
Thus, control input constraints should be imposed during the consensus procedure.

It is well known that the receding horizon control (RHC) strategy, also known as model predictive control
is capable of handling system constraints while preserving (sub-)optimal control performance,
and this motivates us to study the constrained consensus problem in an RHC-based framework.
In this paper, we consider two classes of discrete-time linear MASs, i.e., MASs with semi-stable and unstable subsystems.
For both classes of MASs, we first investigate the inverse optimal consensus problem and design optimal consensus protocols.
The main tool for the optimal consensus protocol design is the concept of inverse optimality and set stability.
Based on the designed protocols, we further study the RHC-based consensus problems and investigate the feasibility issue and analyze the achieved consensus property. The main idea utilized in this part are the optimality principle and set stability.

The closely-related literature is reviewed from the following three aspects: 1) Constrained consensus,
2) optimality-based consensus protocol design without constraints,
and 3) RHC-based consensus and cooperative control.
Constrained consensus problem is a very challenging issue. Only few results are reported in the literature,
and most of them deal with MASs with simple integrator dynamics. For example, in \cite{Nedic10TAC_constrained_consensus}, the projected consensus algorithm and subgradient algorithm are proposed for consensus estimate for first-order MASs with convex constraints.
In \cite{Lin14TAC_constrained_consensus}, the consensus problem is investigated for multi-integrators with convex constraints and communication delays. In \cite{Wang14SCL_synchronization}, the synchronization problem of MASs with homogeneous linear dynamics and input saturation is studied
and the synchronization is proved by showing the semi-global stability of error dynamics.

Due to the desired feature of optimal control, the optimal consensus protocol design problem has also received a lot of attention.
For example, the distributed LQR problem is investigated for identical decoupled linear systems and the conditions for achieving global optimality are developed in \cite{Borrelli08TAC_LQR}.
The optimal consensus strategy for discrete-time MASs is proposed in \cite{Johansson08Auto_consensus},
where a negotiation strategy is utilized. In \cite{Cao10TCYB_optimal_consensus},
the LQR-based consensus problem is investigated for multi-integrators by using the Laplacian matrix as the variable in optimization, and the interaction-free and interaction-related cost functions are formulated. In \cite{Hengster14TAC_consensus_optimal}, the inverse optimality is utilized for solving the consensus and synchronization problems for continuous-time MASs.
It is shown by all the aforementioned results that the optimal control performance index is generally coupled with the network topology.

In the literature of RHC strategy for MASs, most of the results have been contributed for cooperative stabilization problems, for example,
\cite{Dunbar_06_Auto_DMPC,Li_Auto_14_DRHC,Franco_08_TAC_DMPC,Richards_IJC_07_RDMPC,Li_TAC_14_RDRHC,Muller_12_DMPC_IJRNC}.
Unlike the cooperative stabilization problem, the consensus problem needs to take special attention to deal with information contained
in network topology. Thus, the RHC-based consensus problem is more challenging, and most of consensus strategies are developed
for MASs with simple dynamics.
In \cite{Ferrari_09_TAC_MPC_consensus}, the RHC-based consensus strategies are proposed for MASs with integrator and double-integrator dynamics, where the concept of optimal path is utilized to prove consensus property. However, this method may not be directly applicable for MASs with higher order dynamics. In \cite{Zhan_13_auto_consensus}, the consensus problem for MASs with integrator is solved by using unconstrained RHC, where multiple-time information needs to be exchanged at each time instant. In \cite{Zhang15TCSI}, the RHC-based consensus problem is studied for MASs with double-integrator and input constraints. The RHC-based consensus problem is investigated for MASs with general linear dynamics in \cite{Huiping14Auto_Consensus}, but the constraints are not considered.

In this paper, we propose a solution to the RHC-based consensus problem for general linear MASs with input constraints.
The main contributions of this paper are three-fold:
\begin{itemize}
\item The global optimal consensus protocols and the conditions for designing such protocols are proposed
      for MASs with semi-stable and unstable subsystems. It is shown that the global optimal performance indices
      are dependent on the network topology and the system dynamics, indicating the difficulty for designing optimal
      cost functions by direct approaches. The developed results not only offer an approach to design optimal consensus protocols
      for unconstrained MASs, but alos provide a way of designing auxiliary consensus protocols for constrained MASs.

\item Novel centralized RHC-based consensus strategies that can fulfill control input constraints are developed for both classes of MASs, where the design of terminal costs and constraints are built on the developed optimal consensus protocols. The conditions for decomposition of cost functions and constraints are provided, based on which the distributed RHC-based consensus strategies are designed for MASs with constraints.

\item The iterative feasibility is proven for the designed RHC-based consensus strategy, and the consensus properties are analyzed for both classes of MASs. We prove that, for the MASs with semi-stable subsystems, the control input constraints are satisfied and the closed-loop system reach {\em convergent consensus}, but for MASs with unstable subsystems, only conventional consensus is guaranteed.
\end{itemize}

The remainder of this paper is organized as follows.
The graph notations and preliminary results on set stability are presented in Section \ref{Sec_notaitons}.
The problem formulation and inverse optimal controller design are given in Section \ref{sec_problem_formulation}.
In Section \ref{Sec_semistable}, the results for designing optimal consensus protocols for MASs with semi-stable subsystems are first developed,
then the RHC-based consensus strategies are presented, and finally, the feasibility and consensus analysis of the closed-loop systems
are provided. The parallel results for MASs with unstable subsystems are given in Section \ref{Sec_General_linear}.
The design conditions for cost functions and imposed network constraints are discussed in Section \ref{Sec_discussions}.
In Section \ref{Sec_simulation}, two numeric examples are provided, and the conclusion remarks are summarized in Section \ref{Sec_conclusion}.

Notation: The superscripts ``${\rm T}$'', ``$-1$'' and ``$\#$'' are denoted by the matrix transposition,
inverse and group inverse, respectively. $\mathbb{R}$ and $\mathbb{Z}$ represent the real numbers and integers, respectively.
$\mathbb{R}_{\geqslant 0}$ and $\mathbb{Z}_{+}$ are denoted by the nonnegative real numbers and integers.
Given a matrix $P$, we use $P>0$ ($P\geqslant0$) to denote its positive-definiteness (semi-positive definiteness).
Given a vector $x\in\mathbb{R}^n$, its Euclidean norm is denoted by $\|x\|$, and its $P$-weighted norm by $\|x\|_{P} \triangleq \sqrt{x^{\rm T}Px}$,
where $P \geqslant 0$.
The distance between $x$ and a set $\mathcal{O}\subseteq\mathbb{R}^n$ is denoted by $|x|_{\mathcal{O}} = \inf_{y\in\mathcal{O}}\|x-y\|$.
Give a matrix $P$, we use $\lambda(P)$ to represent its eigenvalue, ${\rm spec}(P)$ to stand for its spectrum radius,
and $\sigma_{\min}(P)$ and $\sigma_{\max}(P)$ to denote its minimum and maximum nonzero spectrum, respectively.
For a matrix $A$, its range and null space are denoted by ${\rm Ker }(A)$ and ${\rm range}(A)$, respectively.
We write the column operation $[x_1^{\rm T},x_2^{\rm T},\cdots,x_n^{\rm T}]^{\rm T}$ as ${\rm col}(x_1,x_2,\cdots,x_n)$.
Given two sets $A \subseteq B\subseteq \mathbb{R}^{n}$, the difference between the two sets is
defined by $A\setminus B \triangleq \{x| x \in A, x\notin B\}$.
Finally, we use $\otimes$ to denote the Kronecker product operation.

\section{Graph Theory and Preliminary Results}\label{Sec_notaitons}
\subsection{Graph Theory and Notations}
A graph $\mathcal{G}$ is characterized by a triple $\{\mathcal{V}, \mathcal{E}, \mathcal{A}\}$,
where $\mathcal{V} = \{1, \cdots, N\}$ represents the collection of $N$ vertices (nodes),
$\mathcal{E}\subseteq \mathcal{V}\times \mathcal{V}$ is the set of arcs or edges, and
$\mathcal{A} = [a_{ij}]\in\mathbb{R}^{n\times n}$ with $a_{ij}\geqslant 0$ is the weighted adjacency matrix of the graph $\mathcal{G}$.
The edge is represented by a pair $(j,i)$, $j \in\mathcal{V}$ and $i \in\mathcal{V}$, and it is assumed that
the graph contains no-self loop, that is $(i,i)\notin \mathcal{E}$.
A edge $(j,i)$, $i\neq j$ means that there is a communication channel from node $j$ to $i$.
The neighbors of node $i$ are denoted by $\mathcal{N}_i = \{j\in \mathcal{V}|(j,i)\in \mathcal{E}\}$.
For simplicity, the weights in the adjacency matrix are set to be $1$, i.e., $a_{ij}=1$ if $(j,i)\in\mathcal{E}$;
otherwise, $a_{ij} = 0$. Define the in-degree of node $i$ by $deg_i = \sum_{j=1}^{N} a_{ij}$, and the degree matrix
of graph $\mathcal{G}$ is defined by $\mathcal{D} = {\rm diag}(deg_1,\cdots, deg_N)$.
The graph Laplacian matrix is defined by $\mathcal{L} = \mathcal{D} -\mathcal{A}$.
A graph is called undirected if $a_{ij}>0 \iff a_{ji}>0$; otherwise, the graph is called a directed graph (diagraph).
A path from node $i_1$ to $i_k$ is denoted by a sequence $(i_1, i_2), (i_2,i_3), \cdots, (i_{k-1}, i_{k})$, where $(i_{j-1}, i_{j})$ or
$(i_{j}, i_{j-1})\in\mathcal{E}$ with $j=2, \cdots, k$.
A directed path from node $i_1$ to $i_k$ is denoted by a sequence $(i_1, i_2), (i_2,i_3), \cdots, (i_{k-1}, i_{k})$ with $(i_{j-1}, i_{j})\in\mathcal{E}$.

A diagraph is called strongly connected if any two vertices can be connected by a directed path.
A diagraph is said to contain a spanning tree, if there exists a node $i$ such that it can connect to every other node in $\mathcal{V}$ via
a directed path, and the node $i$ is called the root node. A detailed balanced graph is a graph satisfying $\lambda_i a_{ij} = \lambda_j a_{ji}$ for some positive constants $\lambda_1,\cdots, \lambda_M$. A diagraph contains a simple Laplacian if the eigenvalues of its Laplaian matrix are simple.
The Laplacian matrix has a simple zero eigenvalue if and only if a diagraph contains a spanning tree or an undirected graph is connected.
In particular, for an undirected graph, the Laplacian matrix is positive semi-definite and the eigenvalues can be arranged by an ascending order as
$0= \lambda_1(\mathcal{L})\leqslant \lambda_2(\mathcal{L})\leqslant \cdots \leqslant \lambda_N(\mathcal{L})$.

\subsection{Preliminary Results for Set Stability}
Consider a discrete-time system
\begin{equation}\label{equ_nonlinear}
x_{k+1} = f(x_{k}), k\in \mathbb{Z}_{+},
\end{equation}
where $x_k \in\mathbb{R}^{n}$, and $f:\mathbb{R}^{n} \rightarrow \mathbb{R}^{n}$, is continuous.
The solution to (\ref{equ_nonlinear}) is denoted by $x(k,x_0)$ with the initial state $x_0$.
Denote $\mathcal{O}$ by a nonempty closed subset of $\mathbb{R}^n$, and $\mathcal{O}$ is not necessarily compact.
The set $\mathcal{O}$ is said to be forward invariant for the system in (\ref{equ_nonlinear}),
if for any $x_{0} \in \mathcal{O}$, it follows that $x(k,x_{0}) \in \mathcal{O}$, for any $k \geqslant 0$.

Motivated by the set stability definition in \cite{Jiang02SCL_inverse_set_stability}, we present the definition of asymptotic stability as follows.
Before that, some definitions for four classes of functions are recalled.
\begin{definition}\label{def_functions}
A function $\mu:\mathbb{R}_{\geqslant 0} \rightarrow \mathbb{R}_{\geqslant 0}$ is said to be a $\mathcal{K}$-function,
if $\mu(0) = 0$, and it is continuous and strictly increasing; furthermore, if $\lim_{t\rightarrow \infty}\mu(t) = \infty$,
then $\mu$ is called a $\mathcal{K}_\infty$-function. A function $\gamma:\mathbb{R}_{\geqslant 0} \rightarrow \mathbb{R}_{\geqslant 0}$ is said to be a positive function, if $\mu(t)>0$ for all $t>0$ and $\mu(t) = 0$ for $t = 0$. A function $\beta: \mathbb{R}_{\geqslant 0}\times \mathbb{R}_{\geqslant 0} \rightarrow \mathbb{R}_{\geqslant 0}$ is said to be a $\mathcal{KL}$-function, if for any fixed $x\geqslant 0$, $\beta(x,.)$ is decreasing and $\lim_{t\rightarrow \infty} \beta(x,t) = 0$, and for any fixed $t\geqslant 0$, $\beta(.,t)$ is a $\mathcal{K}$-function.
\end{definition}

\begin{definition}\label{def_AS}
For the system in (\ref{equ_nonlinear}), suppose that there is a forward invariant set $\mathcal{O}$.
It is said to be asymptotically stable with respect to the set $\mathcal{O}$, if the following
two conditions hold:
\begin{itemize}
\item [1)] Lyapunov stability: for every $\epsilon>0$, there exists some $\delta>0$ such that,
\begin{equation*}
|x_0|_{\mathcal{O}} < \delta \Rightarrow |x(k,x_0)|_{\mathcal{O}} < \epsilon, \forall k\geqslant 0.
\end{equation*}
\item [2)] Attraction: for $x_0\in\mathcal{X}\subseteq \mathbb{R}^{n}$, $\lim_{k \rightarrow \infty} |x(k,x_0)|_{\mathcal{O}} = 0$.
\end{itemize}
\end{definition}

\begin{theorem}\cite{Jiang02SCL_inverse_set_stability}\label{thm_set_stability}
For the system in (\ref{equ_nonlinear}) with a given forward invariant set $\mathcal{O}\in\mathbb{R}^{n}$, if there exists a continuous function $V:\mathbb{R}^n\rightarrow \mathbb{R}_{\geqslant 0}$, such that
\begin{itemize}
\item [1)] $\alpha_1(|x|_{\mathcal{O}})\leqslant V(x)\leqslant \alpha_2(|x|_{\mathcal{O}})$,
\item [2)] $V(f(x))-V(x)\leqslant -\alpha_3(|x|_{\mathcal{O}})$,
\end{itemize}
for any $x\in\mathcal{X} \subseteq \mathbb{R}^n$, where $\alpha_1$ and $\alpha_2$ are $\mathcal{K}$-function, and $\alpha_3$ is a
positive function, then the system in (\ref{equ_nonlinear}) is asymptotically stable with respect to the set $\mathcal{O}$.
\end{theorem}

\section{Problem Formulation}\label{sec_problem_formulation}
Consider an MASs with $M$ agents and the dynamics of each agent $i$ is
\begin{equation}\label{equ_agent}
x_{k+1}^i = A x_k^i + B u_k^i,
\end{equation}
where $x_{k}^i\in\mathbb{R}^n$ is the state, $u_k^i\in\mathbb{R}^m$ is the control input.
The control input is required to fulfill the constraint as
\begin{equation}\label{equ_constraint}
u_k^i \in \mathcal{U}_i,
\end{equation}
where $\mathcal{U}_i$ are compact sets, and contain the origin as their interior points.
Each agent $i$ can communicate with some neighboring agents via the communication network, which is characterized by
a graph $\mathcal{G}$.

The overall augmented system can be written as
\begin{equation}\label{equ_overall_system}
X_{k+1} = (I_M \otimes A) X_k + (I_M\otimes B) U_k,
\end{equation}
where $X_k = {\rm col}(x_k^1, \cdots, x_k^M)$, and $U_k = {\rm col}(u_k^1, \cdots, u_k^M)$.
The system constraint becomes $U_k \in \mathcal{U}$,
where $\mathcal{U} = \mathcal{U}^1\times\cdots\times \mathcal{U}^M$.
\begin{definition}\label{def_consensus}
For the MAS characterized by (\ref{equ_agent}) and the communication graph $\mathcal{G}$, with certain control input $u_k^i$ to close the loop,
it is said to reach consensus, if
\begin{equation*}
\lim_{k\rightarrow \infty} \|x_k^i-x_k^j\| = 0, \forall i,j = 1,\cdots, M.
\end{equation*}
Furthermore, if it reaches consensus, and $\lim_{k\rightarrow \infty} \|x_k^i\|<\infty$, $\forall i = 1,\cdots, M$,
then the MAS is said to reach convergent consensus.
\end{definition}


A necessary property for the system in (\ref{equ_agent}) to achieve consensus is assumed as follows \cite{Ma10TAC_IFF_consensus}.
\begin{assumption}\label{ass_controlability}
The pair $(A, B)$ is controllable.
\end{assumption}

To focus on our main objectives, we restrain our attention to the case of fixed graphs.
Due to the fact that having a spanning tree is a necessary
condition to achieve consensus for general linear MASs, we make the following assumption.
\begin{assumption}\label{ass_spanning_tree}
The graph $\mathcal{G}$ contains a spanning tree.
\end{assumption}

According to Definition \ref{def_consensus}, the MAS in (\ref{equ_agent}) achieves consensus, meaning that
the state for each agent will eventually converge to the consensus set $\mathcal{C}\triangleq \{x^1 = x^2 =\cdots = x^M\}$.
By Definition \ref{def_AS}, the asymptotic stability with respect to the set $\mathcal{C}$ for the closed-loop system in (\ref{equ_overall_system})
ensures the state $X_k$ will enter the set $\mathcal{C}$ when $k\rightarrow \infty$,
implying that the MAS in (\ref{equ_agent}) reaches consensus.
As a result, we have the following result.
\begin{theorem}\label{thm_AS_Consensus}
If the closed-loop system in (\ref{equ_overall_system}) under certain control protocol
is asymptotically stable with respect to the set $\mathcal{C}$,
then the states for the MAS will reach consensus.
\end{theorem}

The inverse optimality approach will be utilized to design the RHC-based consensus strategy.
To that end, we first establish a result for general discrete-time linear systems on how to design a control law such that
it is inverse optimal with respect to certain performance index, and that the closed-loop system is set stable.
The results are given in the following Lemma \ref{lemma_inverse_optimality}. Consider a discrete-time linear system:
\begin{equation}\label{equ_linear_system}
x_{k+1} = A x_k + B u_k.
\end{equation}
\begin{lemma}\label{lemma_inverse_optimality}
For the system in (\ref{equ_linear_system}), if there exist a constant $\gamma >0$, and three symmetric matrices $P\geqslant 0$,
$Q \geqslant 0$ and $R>0$, such that \begin{align}
& A^{\rm T} P A -P - A^{\rm T} P B(R + B^{\rm T} P B)^{-1}B^{\rm T} P A + Q =0,\label{equ_ARE}\\
& \|x_k\|_Q \geqslant \gamma |x_k|_{\mathcal{N}}.\label{equ_set_stability}
\end{align}
Then 1) the closed-loop system is asymptotically stable with respect to the set
$\mathcal{N}$; 2) the state feedback controller $u_k = \phi(x_k) \triangleq -(B^{\rm T} P B+R)^{-1} B^{\rm T} P A x_k$ is optimal with
respect to the performance index $J(x_0, u_k) = \sum_{k=0}^{\infty} L(x_k, u_k)$ with $L(x_k, u_k) = \|x_k\|_{Q}^2 + \|u_k\|_R^2$;
3) the optimal performance index $J^*(x_0,\phi(x_k)) = V(x_0)$, where $V(x_k) = x_k^{\rm T} P x_k$ and $\mathcal{N} = {\rm Ker(P)}$.
\end{lemma}

\begin{proof}
Proof of Part 1): With $u_k = \phi(x_k)$, the closed-loop system for the system in (\ref{equ_linear_system})
is $x_{k+1} = [A - B(B^{\rm T} P B+R)^{-1} B^{\rm T} P A] x_k$. Constructing a Lyapunov function $V(x_k)$, one has
$\sigma_{\min}(P)(|x_k|_{\mathcal{N}})^2 \leqslant V(x_k) \leqslant \sigma_{\max}(P)(|x_k|_{\mathcal{N}})^2$ and
$V(x_{k+1}) - V(x_k) = -x^{\rm T}_k(Q + A^{\rm T} P B (B^{\rm T} P B+R)^{-1} $
$R (B^{\rm T} P B+R)^{-1} B^{\rm T} P A) x_k \leqslant
-(\gamma |x_k|_{\mathcal{N}})^{2}$.
Thus, the conditions in Theorem \ref{thm_set_stability} are satisfied,
and the closed-loop system is asymptotically stable with respect to the set $\mathcal{N}$.

Proof of Part 2): Define $J_N(x_0,u_k) = \sum_{k=0}^{N} L(x_k, u_k)$. Since
\begin{equation*}
\sum_{k=0}^N(V(x_{k+1}) - V(x_k)) = V(x_{N}) - V(x_0),
\end{equation*}
one has that
\begin{align*}
& J_N(x_0,u_k)\\
= & \sum_{k=0}^{N} (\|x_k\|^2_{Q} + \|u_k\|^2_{R} + V(x_{k+1}) - V(x_k)) - V(x_N) + V(x_0),\\
= & \sum_{k=0}^{N} (\|x_{k+1}\|^2_P + \|x_k\|^2_{(Q-P)} +\|u_k\|^2_{R}) - V(x_N) + V(x_0).
\end{align*}
Plugging the system dynamics (\ref{equ_linear_system}) into the above equation, we have
\begin{align*}
J_N(x_0,u_k) = & \sum_{k=0}^{N} (\|x_k\|^2_{(Q-P + A^{\rm T} P A)} +\|u_k\|^2_{(R + B^{\rm T} P B)}\\
               & + 2x_k^{\rm T} A^{\rm T} P B u_k) - V(x_N) + V(x_0).
\end{align*}
Using the condition (\ref{equ_ARE}), we get
\begin{align*}
& J_N(x_0,u_k) = \sum_{k=0}^{N} (\|x_k\|^2_{(A^{\rm T} P B(R + B^{\rm T} P B)^{-1}B^{\rm T} P A)} \\
               &+\|u_k\|^2_{(R + B^{\rm T} P B)} + 2x_k^{\rm T} A^{\rm T} P B u_k) - V(x_N) + V(x_0).
\end{align*}
According to \cite{Frank1986}, this summation can be further written in a square form as
\begin{align*}
& J_N(x_0,u_k) = \sum_{k=0}^{N} (\|u_k - \phi(x_k)\|^2_{(R + B^{\rm T} P B)}) - V(x_N) + V(x_0).
\end{align*}
As a result, we have
\begin{align}
J(x_0,u_k) = &\lim_{N\rightarrow \infty} J_N(x_0,u_k),\nonumber \\
            =&\sum_{k=0}^{\infty} (\|u_k - \phi(x_k)\|^2_{(R + B^{\rm T} P B)})\nonumber \\
            &  - \lim_{k\rightarrow \infty}V(x_k) + V(x_0).\label{equ_J_x_u}
\end{align}
From (\ref{equ_J_x_u}), we can see that the performance index $J(x_0,u_k)$ is minimized by $u_k = \phi(x_k)$.

Proof of Part 3): According to Part 1), the closed-loop system is set stable with respect to $\mathcal{N}$, one has
$\lim_{k\rightarrow \infty}V(x_k) = 0$. In terms of (\ref{equ_J_x_u}), the optimal value of $J(x_0,u_k)$ is $V(x_0)$.
The proof is completed.
\end{proof}
\section{Constrained Consensus for Subsystems with Semi-stable Dynamics}\label{Sec_semistable}
This section considers the case that $A$ is semi-stable in the subsystem dynamic (\ref{equ_agent}).
In this situation, we first propose a class of optimal consensus protocols for the MAS to reach consensus.
Then we investigate under what conditions the derived optimal performance index can be distributively assigned among each agent, based on which
we propose a distributed constrained RHC-based consensus strategy. Finally, we analyze the feasibility issue and consensus property
of the designed constrained consensus strategy.
\subsection{Properties of Semi-stable Systems}
The fact that $A$ is semi-stable means that ${\rm spec}(A)\leqslant 1$, and if
${\rm spec}(A)=1$, then $1$ is a simple eigenvalue.
To facilitate presenting the property of the semi-stable system, we recall the definition of semi-observability \cite{Hui09IJC_semistable}
\begin{definition}\label{def_semiobservability}
The pair $(C, A)$ is said to be semi-observable, if ${\rm Ker}(A-I_n) = \bigcap_{i=0}^{n-1} {\rm Ker}(C(A-I_n)^{i})$,
where $A\in\mathbb{R}^{n\times n}$, and $C \in\mathbb{R}^{p\times n}$, and $(A-I_n)^0 = I_n$.
\end{definition}
\begin{lemma}\label{lem_semistability}\cite{Hui09IJC_semistable}
If $A$ is semi-stable, then for every semi-observable pair $(C, A)$, there exists a symmetric matrix $S>0$, such that
\begin{equation}\label{equ_Lyapunov_equation}
A^{\rm T} S A - S = C^{\rm T} C.
\end{equation}
Furthermore, $S$ can be taken as in the form
\begin{equation}\label{equ_S_value}
S = \sum_{i=0}^{\infty} (A^{k})^{\rm T} C^{\rm T} C A^k + a L^{\rm T} L,
\end{equation}
where $a>0$ is a constant, and $ L = I_n - (A-I_n)(A-I_n)^{\#}$.
\end{lemma}
\subsection{Optimal Consensus Protocol}
For the MAS in (\ref{equ_agent}), it is known that a class of consensus protocols can be taken as \cite{Ma10TAC_IFF_consensus}\cite{You_11_TAC_consensus}
\begin{equation}\label{eqn_conventional_protocol}
u_k^i = c\sum_{j\in\mathcal{N}_i}K_2(x_k^i-x_k^j),
\end{equation}
where $c>0$ is the coupling gain and $K_2\in\mathbb{R}^{m\times n}$, is the gain matrix.
As a result, the overall control input $U_k$ becomes $U_k = c (\mathcal{L} \otimes K_2) X_k$,
where $\mathcal{L}$ is the Laplacian matrix of the graph $\mathcal{G}$.
There are many ways to design $K_2$ to drive the MAS to reach consensus, for example, \cite{You_11_TAC_consensus}\cite{Hengster13Auto_synchronization}.
In the following, we propose a way to design $K_2$ such that the consensus can be reached,
and the resultant overall control $U_k$ is optimal with respect to a global performance index for the overall system in (\ref{equ_overall_system}).
This consensus protocol will facilitate the design of the constrained RHC consensus strategy.
\begin{theorem}\label{thm_consensus_inverse_semi}
For the system in (\ref{equ_agent}), assume that $B$ is of full column rank. If the consensus gain is designed as $K_2 = -(B^{\rm T} S_2 B+R_2)^{-1} B^{\rm T} S_2 A$, then the consensus can be reached, and the control input $U_k = c (\mathcal{L} \otimes K_2) X_k$ is optimal with respect to the performance index $J_s(X_0,U_k) = \sum_{k=0}^{\infty} \|X_k\|_{Q_s}^2 + \|U_k\|_{R_s}^2$ for the overall system in (\ref{equ_overall_system}), with
$Q_s = S_1\otimes Q_2 + \frac{c S_1 \mathcal{L}}{1+\alpha}\otimes H$, $H = A^{\rm T} S_2 B(B^{\rm T} S_2 B)^{-1}B^{\rm T} S_2 A$,
and $R_s =  R_1\otimes R_2$, where the parameters are designed as follows:
1) $Q_2 = C_2^{\rm T} C_2$ with $(C_2, A)$ being a semi-observable pair and ${\rm rank}(C_2) = n-1$; 2) $S_2$ is a symmetric and positive definite solution to (\ref{equ_Lyapunov_equation}); 3) $S_1 = W \mathcal{L}$, with $W$ being symmetric and invertible, and $W \mathcal{L}$ and $W \mathcal{L}^{2}$ being symmetric; 4) $R_1 = \frac{W(I_M -c \mathcal{L})}{c \alpha}$; 5) $R_2 = \alpha B^{\rm T} S_2 B$, where $\alpha>0$ is a constant; 6) $c$ is designed such that $c\leqslant \frac{1}{\sigma_{\max}(\mathcal{L})}$.
\end{theorem}
\begin{proof}
Define $S_s = S_1 \otimes S_2$ and consider a term for the overall system in (\ref{equ_overall_system})
\begin{align*}
\Sigma \triangleq & (I_M \otimes A)^{\rm T} S_s (I_M \otimes A) - S_s  -[(I_M \otimes A)^{\rm T} S_s (I_M \otimes B)]\\
& \times [R +(I_M \otimes B)^{\rm T} S (I_M \otimes B)]^{-1}[(I_M \otimes B)^{\rm T} S_s (I_M \otimes A)].
\end{align*}
Using the property of the Kronecker product, one has
\begin{align*}
\Sigma  =&  S_1 \otimes A^{\rm T} S_2 A - S_1\otimes S_2 - (S_1 \otimes A^{\rm T} S_2 B)\\
& \times[S_1 \otimes B^{\rm T} S_2 B + R_1 \otimes R_2]^{-1} (S_1 \otimes B^{\rm T} S_2 A).
\end{align*}
Applying the design condition in (\ref{equ_Lyapunov_equation}) for $S_2$ and using $R_2 = \alpha B^{\rm T} S_2 B$,
we have
\begin{align*}
\Sigma =& -S_1 \otimes Q_2 - S_1(S_1 + \alpha R_1)^{-1} S_1 \otimes H.
\end{align*}
Note that the design conditions 3), 4) and 6) ensure $(S_1 + \alpha R_1)$ is invertible.
Using the conditions 3) and 4), we have
\begin{equation}\label{equ_inermidiate_S1}
 (S_1 + \alpha R_1)^{-1} S_1= \frac{c\mathcal{L}}{1+\alpha}.
\end{equation}
As a result, $\Sigma = - Q_s$. That is, the ARE in (\ref{equ_ARE}) is satisfied.
Furthermore, the consensus protocol $U_k = K^* X_k \triangleq c \mathcal{L}\otimes K_2 X_k$ can be further written as
\begin{align*}
K^* = -\frac{c\mathcal{L}}{1+\alpha} \otimes (B^{\rm T} S_2 B)^{-1} B^{\rm T} S_2 A,
\end{align*}
where the condition in 5) is utilized. Using (\ref{equ_inermidiate_S1}), one obtains that
$K^* =(S_1 + \alpha R_1)^{-1} S_1 \otimes (B^{\rm T} S_2 B)^{-1} B^{\rm T} S_2 A$.
Applying the property of the Kronecker product, one gets
$K^* =(S_1 \otimes B^{\rm T} S_2 B +  R_1\otimes R_2)^{-1} (S_1 \otimes (B^{\rm T} S_2 B)^{-1} B^{\rm T} S_2 A)$,
which is further equivalent to
\begin{align*}
K^* = & -[(I_M \otimes B)^{\rm T} S_s (I_M \otimes B) + R ]^{-1}\\
      & \times (I_M \otimes B)^{\rm T} S_s(I_M \otimes A).
\end{align*}
According to Lemma \ref{lemma_inverse_optimality}, the control protocol $U_k = K^* X_k$ is indeed optimal with respect to
the performance index $J_s(X_0, U_k)$.

Next, we need to prove that $\|X_k\|_{Q} \geqslant \gamma_1 \|X_k\|_{\mathcal{N}}$, where $\gamma_1>0$ is some constant, and $\mathcal{N} ={\rm Ker}(S_s)$. Since $S_1 = W \mathcal{L}$ and $W$ is invertible, and $S_2$ is invertible, one obtains that $\mathcal{N} = {\rm Ker}(\mathcal{L}\otimes I_n)$. Because $Q_2\geqslant 0$, the null space of $S_1 \otimes Q_2$ can be represented by the union of two sets, $\mathcal{N}$ and $\mathcal{N}_1 = {\rm Ker}(I_M \otimes Q_2)$, i.e., ${\rm Ker}(S_1 \otimes Q_2) = \mathcal{N} \cup \mathcal{N}_1$. Define $\bar{\mathcal{N}}_1 = \mathcal{N} \cap \mathcal{N}_1$. Note that $\bar{\mathcal{N}}_1$ is not an empty set. Similarly, the null space of the matrix $\frac{c S_1 \mathcal{L}}{1+\alpha}\otimes H$ can also be made up from two parts, i.e., ${\rm Ker}(\frac{c S_1 \mathcal{L}}{1+\alpha}\otimes H) = \mathcal{N} \cup \mathcal{N}_2$, where $\mathcal{N}_2 = {\rm Ker }(I_M \otimes H)$. Define $\bar{\mathcal{N}}_2 = \mathcal{N} \cap \mathcal{N}_2$.

Firstly, we prove the fact that $(\mathcal{N}_1 \setminus \bar{\mathcal{N}}_1) \cap (\mathcal{N}_2 \setminus \bar{\mathcal{N}}_2) = {0}$. This is proved by contradiction.
Assume that there is an element $v_1 \neq 0$ in $\mathcal{N}_1\setminus \bar{\mathcal{N}}_1$ and it also belongs to $\mathcal{N}_2 \setminus \bar{\mathcal{N}}_2$. Denote the corresponding eigenvector for the eigenvalue $0$ of $Q_2$ by $w_1$. Then $\mathcal{N}_1\setminus \bar{\mathcal{N}}_1$ can be represented by ${\rm span}(e_i\otimes w_1)$, where $i=1,\cdots,M-1$.
Without loss of generality, take $v_1 = e_1 \otimes w_1$. Since $Q_2 w_1 = 0$, it follows that $C_2 w_1 = 0$. By the condition that $(C_2, A)$ is semi-observable, one gets $A w_1 = w_1$. On the other hand, $v_1 \in \mathcal{N}_2 \setminus \bar{\mathcal{N}}_2$ implies $\frac{c S_1 \mathcal{L}}{1+\alpha}e_1 \otimes H w_1 = 0$. It is noted that $\mathcal{L} e_1 \neq 0$. As a result, it is required that $Hw_1 = 0$.
That is equivalent to $w_1^{\rm T}A^{\rm T} S_2 B(B^{\rm T} S_2 B)^{-1}B^{\rm T} S_2 A w_1 = 0$. Using the fact that $A w_1 = w_1$,
it is further required that $w_1^{\rm T} S_2 B(B^{\rm T} S_2 B)^{-1}B^{\rm T} S_2 w_1 = 0$. Since $S_2 >0$ and $B^{\rm T} S_2 B>0$, the requirement
is equivalent to $B^{\rm T} S_2 w_1 = 0$.
According to the design of $S_2$ in (\ref{equ_S_value}), one has $B^{\rm T}(\sum_{k=0}^{\infty} (A^{k})^{\rm T} Q_2 A^k + a L^{\rm T} L) w_1 = 0$.
Using the fact that $A w_1 = w_1$ and $Q_2 w_1 = 0$, it follows that $B^{\rm T}L^{\rm T} L w_1 = 0$.
Note that $w_1 \in {\rm Ker }(A - I_n)$, as a result, $L w_1 = w_1$, leading to
\begin{equation}\label{equ_w1_result1}
B^{\rm T}L^{\rm T} w_1 = 0.
\end{equation}
Since $L(A-I_n) = (A - I_n) - (A-I_n)(A-I_n)^{\rm \#}(A-I_n) = 0$, we obtain $ L = L A$. Plugging this into (\ref{equ_w1_result1}),
one has $B^{\rm T} A ^{\rm T}L^{\rm T} w_1 = 0$. Similarly, we can obtain $B^{\rm T} (A^2)^{\rm T}L^{\rm T} w_1 = 0, \cdots, B^{\rm T} (A^{(n-1)})^{\rm T}L^{\rm T} w_1 = 0$. As a result, we get
$\left[\begin{array}{c}
   B^{\rm T}\\
   B^{\rm T} A ^{\rm T} \\
   \vdots \\
   B^{\rm T} (A^{(n-1)})^{\rm T}
 \end{array}\right] L^{\rm T} w_1  = 0$. Because $(A,B)$ is controllable, ${\rm rank}(B, AB, \cdots, A^{(n-1)}B) = n$.
Therefore, we require $L^{\rm T} w_1 = 0$, implying $w_1 \in {\rm Ker}(L^{\rm T})$. On the other hand, according to Theorem 2.1 in \cite{Hui09IJC_semistable},
${\rm range}(L) = {\rm Ker}(A-I_n)$, indicating that $w_1 \in {\rm range}(L)$. As a result, it follows that $w_1 = 0$.
This contradicts with $v_1 = e_1 \otimes w_1 \neq 0$.

Since we have proved that $(\mathcal{N}_1 \setminus \bar{\mathcal{N}}_1) \cap (\mathcal{N}_2 \setminus \bar{\mathcal{N}}_2) = {0}$, implying that
the null space of $Q_s$ is $\mathcal{N} = {\rm Ker}(\mathcal{L}\otimes I_n)$. As a result, we have $\|X\|_{Q_s} \geqslant \sqrt{\sigma_{\min}(Q_s)} |X|_{\mathcal{N}}$. Hence, the condition in (\ref{equ_set_stability}) holds.
Applying Lemma \ref{lemma_inverse_optimality}, it follows that the closed-loop system is asymptotically stable with respect to the set $\mathcal{N}$,
implying that the consensus is reached by the consensus protocol (\ref{eqn_conventional_protocol}) with $K_2$ being designed as in the theorem.
The proof is completed.
\end{proof}
\subsection{RHC-based Consensus Strategy}
In this subsection, the design of the terminal constraint is firstly presented, then the RHC-based consensus strategy is designed.
After that, the condition that can make the optimal cost function equivalently be assigned to each agent is developed, based on which the distributed RHC-based consensus strategy is finally stated.
\subsubsection{Terminal Constraint}
For the overall system in (\ref{equ_overall_system}) with the optimal state feedback $U_k = K^* X_k$, the closed-loop system becomes
\begin{equation}\label{equ_closed_loop_system1}
X_{k+1} = [I_M \otimes A + (I_M \otimes B) K^*] X_k.
\end{equation}
For the system in (\ref{equ_closed_loop_system1}), given a parameter $\beta>0$, define the level set with respect to the set $\mathcal{N}$ as
$\mathcal{O}_{\beta} = \{X_k\in\mathbb{R}^{Mn}: \|X_k\|_{S} \leqslant \beta\}$. Note that $\mathcal{O}_{\beta}$ is closed but not necessarily compact.
\begin{lemma}\label{lemma_level_set}
For any given $\beta>0$, the level set $\mathcal{O}_{\beta}$ with respect to $\mathcal{N}$ is forward invariant for the system in (\ref{equ_closed_loop_system1}).
\end{lemma}
\begin{proof}
Using the property of the ARE, one gets that $\|X_{k+1}\|_S^2 - \|X_k\|_S^2 = -\|X_k\|^2_{((K^*)^{\rm T} R K^* + Q)}\leqslant -\sigma_{\min}(Q)(|X_k|_{\mathcal{N}})^2\leqslant 0$. As a result, $\forall X_k \in \mathcal{O}_\beta$, it implies $X_{k+1} \in \mathcal{O}_\beta$.
The proof is completed.
\end{proof}

\begin{lemma}\label{lemma_terminal_set}
For the system in (\ref{equ_closed_loop_system1}) with constraint $U_k = K^* X_k \in\mathcal{U}$, there exists a $\beta_s>0$ such that
$X_0 \in \mathcal{O}_{\beta_s}$ implies $X_k \in \mathcal{O}_{\beta_s}$ and $U_k \in \mathcal{U}$, for all $k\geqslant 0$.
\end{lemma}
\begin{proof}
Since $\mathcal{U}$ is compact and contains the origin as its interior point, and $\mathcal{N}\subseteq {\rm Ker}(K^*)$, it follows that there
exists an $\epsilon_1 >0$ such that $X \in \mathcal{O}_{\epsilon_1}$ implies $ K^* X \in \mathcal{U}$. According to Lemma \ref{lemma_level_set},
$\mathcal{O}_{\epsilon_1}$ is forward invariant for the system in (\ref{equ_closed_loop_system1}).
As a result, $X_0 \in \mathcal{O}_{\epsilon_1}$ implies $X_k \in \mathcal{O}_{\epsilon_1}$, for all $k\geqslant 0$, and further indicates $U_k \in\mathcal{U}$. Thus, $\beta_s$ can be taken as $\epsilon_1$. This completes the proof.
\end{proof}

In what follows, the set $\mathcal{O}_{\beta_s}$ will be chosen as the terminal set to impose terminal constraint as conventional RHC strategy. Note that the set $\mathcal{O}_{\beta_s}$ should be designed as large as possible to reduce conservatism in RHC algorithm.
Theoretically, $\beta_s$ can be calculated by $\beta_s = \max_{X} \{\epsilon| K^*X \in U, X \in \mathcal{O}_{\epsilon}\}$.

\subsubsection{RHC-based Consensus Strategy}
For the system in (\ref{equ_overall_system}), define an optimization problem as
\begin{problem}\label{problem_CRHC_Semi}
$\min \bar{U}_k^* = {\rm arg} J_s(X_k,\bar{U}_k), {\text subject \; to}$
\begin{align*}
X_{k+i+1|k} = &(I_M \otimes A) X_{k+i|k} + (I_M \otimes B) U_{k+i|k},\\
& U_{k+i|k}\in\mathcal{U}, X_{k+N|k}\in\mathcal{O}_{\beta_s},
\end{align*}
where $i=0,\cdots, N-1$, $X_{k|k} = X_k$, and $\bar{U}_k = {\rm col}(U_{k|k},\cdots, U_{k+N-1|k})$.
\end{problem}
The cost function is defined as
\begin{align*}
J_s(X_k,\bar{U}_k) = \sum_{i=0}^{N-1}\|X_{k+i|k}\|_{Q_s}^2 + \|U_{k+i|k}\|^2_{R_s} + \|X_{k+N|k}\|^2_{S_s},
\end{align*}
where $Q_s$, $R_s$ and $S_s$ are designed as in Theorem \ref{thm_consensus_inverse_semi}, respectively.

A centralized RHC-based consensus strategy would be: At each time instant $k\geqslant 0$, Problem \ref{problem_CRHC_Semi} is solved
for the overall system in (\ref{equ_overall_system}) to generate the optimal control sequence $\bar{U}_k^*$, and the consensus protocol
takes the first element of $\bar{U}_k^*$, i.e., $U_k = U_{k|k}^*$. We will show that this procedure is feasible with appropriate initial data
and the closed-loop system can reach consensus in the following subsection \ref{subsub_feasibility_semi}.

\subsubsection{Distributed RHC Consensus Strategy}
Problem \ref{problem_CRHC_Semi} is a centralized one, requiring a centralized controller.
In this subsection, we develop conditions to make this optimization problem be distributed associated with each agent.
\begin{lemma}\label{lemma_distribution_Semi}
In the cost function $J_s(X_k,\bar{U}_k)$, if the parameter is designed such that
$W = \mu I_M$, and $\mathcal{L} = \mathcal{L}^{\rm T}$, where $\mu >0$ is a scalar, then
$J_s(X_k,\bar{U}_k)$ can be distributively assigned to each agent $i$ by the following sub-cost function as
\begin{align*}
J_s^i(x^i_k,\bar{u}^i_k) =& \sum_{l=0}^{N-1} \sum_{j\in \mathcal{N}_i}a_{ij}\mu[(x_{k+l|k}^i)^{\rm T} Q_2 (x_{k+l|k}^i-x_{k+l|k}^j)\\
                          & - \frac{\mu}{\alpha} [(u_{k+l|k}^i)^{\rm T} R_2 (u_{k+l|k}^i-u_{k+l|k}^j)]\\
                          & + \frac{c \mu }{1+\alpha}\|\sum_{j\in \mathcal{N}_i}a_{ij}(x_{k+l|k}^i-x_{k+l|k}^j)\|_{H}^2 \\
                          & + \frac{\mu}{c\alpha} \|u_{k+l|k}^i\|_{R_2}^2\\
                          & + \mu \sum_{j\in \mathcal{N}_i}a_{ij}[(x_{k+N|k}^i)^{\rm T} S_2 (x_{k+N|k}^i-x_{k+N|k}^j).
\end{align*}
That is, $J_s(X_k,\bar{U}_k) = \sum_{i=1}^{M}J_s^i(x^i_k,\bar{u}^i_k)$.
\end{lemma}
\begin{proof}
Considering the first term $\sum_{j\in \mathcal{N}_i}a_{ij}\mu(x_{k+l|k}^i)^{\rm T} Q_2 (x_{k+l|k}^i-x_{k+l|k}^j)$, one has
\begin{align*}
& \sum_{j\in \mathcal{N}_i}a_{ij}\mu(x_{k+l|k}^i)^{\rm T} Q_2 (x_{k+l|k}^i-x_{k+l|k}^j)\\
=& (x_{k+l|k}^i)^{\rm T}\sum_{j=1}^{M}l_{ij}Q_2x_{k+l|k}^j,
\end{align*}
where $l_{ij}$ is the $(i,j)$-th element of $\mathcal{L}$. As a result,
\begin{align*}
& \sum_{i=1}^{M}\sum_{j\in \mathcal{N}_i}a_{ij}\mu(x_{k+l|k}^i)^{\rm T} Q_2 (x_{k+l|k}^i-x_{k+l|k}^j)\\
=&\|X_{k+l|k}\|_{[\mu(\mathcal{L}\otimes Q_2)]}^2.
\end{align*}
Similarly, we have
\begin{align*}
& \sum_{i=1}^{M}\sum_{j\in \mathcal{N}_i}a_{ij}\frac{\mu}{\alpha}(u_{k+l|k}^i)^{\rm T} R_2 (u_{k+l|k}^i-u_{k+l|k}^j)\\
=&\|U_{k+l|k}\|_{[\frac{\mu}{\alpha}(\mathcal{L}\otimes R_2)]}^2,
\end{align*}
And that
\begin{align*}
& \sum_{i=1}^{M}\sum_{j\in \mathcal{N}_i}a_{ij}\mu(x_{k+N|k}^i)^{\rm T} S_2 (x_{k+N|k}^i-x_{k+N|k}^j)\\
=&\|X_{k+l|k}\|_{[\mu(\mathcal{L}\otimes S_2)]}^2.
\end{align*}
Furthermore, we get
\begin{align*}
&\sum_{i=1}^{M} \frac{c \mu }{1+\alpha} \|\sum_{j\in \mathcal{N}_i}a_{ij}(x_{k+l|k}^i-x_{k+l|k}^j)\|_{H}^2\\
=& \sum_{i=1}^{M} \frac{c \mu }{1+\alpha} [(\mathcal{L}_i \otimes I_M) X_{n|k}]^{\rm T} H[(\mathcal{L}_i \otimes I_M) X_{n|k}]\\
= &\frac{c \mu }{1+\alpha} X_{n|k}^{\rm T} (\mathcal{L} \otimes I_M)^{\rm T} (I_M \otimes H)(\mathcal{L} \otimes I_M) X_{n|k}\\
= & \|X_{n|k}\|^2_{\frac{c S_1 \mathcal{L}}{1+\alpha}\otimes H},
\end{align*}
where $\mathcal{L}_i$ denotes the $i$-th row of $\mathcal{L}$.
By collectively considering above results, we can obtain that $J_s(X_k,\bar{U}_k) = \sum_{i=1}^{M}J_s^i(x^i_k,\bar{u}^i_k)$.
The proof is completed.
\end{proof}

Next, we need to make the constraints in Problem \ref{problem_CRHC_Semi} to be distributively satisfied among agents in the following lemma.
\begin{lemma}\label{lemma_constraint}
For each agent $i$, if the constraints are designed as $u_{k+l|k}^i \in \mathcal{U}^i$, $l=0,\cdots, N-1$, and $\sum_{j\in \mathcal{N}_i}a_{ij}\mu(x_{k+N|k}^i)^{\rm T} S_2 (x_{k+N|k}^i-x_{k+N|k}^j) \leqslant \frac{\beta_s^2}{M}$, then the constraints in Problem \ref{problem_CRHC_Semi} are satisfied.
\end{lemma}

\begin{proof}
Firstly, it can be seen that $u_{k+l|k}^i \in \mathcal{U}^i$, for all $i=1,\cdots, M$ implies $U_{k+l|k}^i \in \mathcal{U}$.
Second, following the similar line of the proof Lemma \ref{lemma_distribution_Semi}, one has $\sum_{i=1}^M \sum_{j\in \mathcal{N}_i}a_{ij}\mu(x_{k+N|k}^i)^{\rm T} S_2 (x_{k+N|k}^i-x_{k+N|k}^j) =\|X_{k+i|k}\|^2_{S_s}\leqslant \beta_s^2$.
This implies $X_{k+N|k} \in \mathcal{O}_{\beta_s}$. The proof is completed.
\end{proof}

Now the optimization problem that is associated with each agent $i$, $i=1\cdots, M$, is formulated as follows:
\begin{problem}\label{problem_semi_distribution}
$\min \bar{u}_k^{i*} = {\rm arg} J_s^i(x_k^i,\bar{u}_k^i), {\text subject \; to}$
\begin{align*}
& x_{k+l+1|k}^i =A x_{k+l|k}^i + B u_{k+l|k}^i,\\
& u_{k+l|k}^i \in\mathcal{U}^i, l=0\cdots, N-1, \\
& \sum_{j\in \mathcal{N}_i}a_{ij}\mu(x_{k+N|k}^i)^{\rm T} S_2 (x_{k+N|k}^i-x_{k+N|k}^j) \leqslant \frac{\beta_s^2}{M},
\end{align*}
where $\bar{u}_{k}^i = {\rm col}(u_{k|k}^i, \cdots, u_{k+N-1|k}^i)$, and $x_k^i = x_{k|k}^i$.
\end{problem}

The distributed RHC-based consensus strategy is summarized as follows: For each agent $i$, at time instant $k$,
it receives information $x_{k+p|k}^j$ and $u_{k+l|k}^j$, $p=0,\cdots N$, $l=0,\cdots, N-1$ from its neighbors via communication network,
then solves Problem \ref{problem_semi_distribution}, and sends its
state information and control information to agents that connect to it.
Finally, the control input is taken as $u_k^i = u_{k|k}^{i*}$.

It can be seen that the distributed RHC strategy is equivalent to the centralized one by appropriately assigning the cost functions and systems constraints as above. So in the following, the performance analysis of the distributed RHC strategy can be executed via the centralized strategy.
\subsection{Feasibility Analysis and Consensus Properties}\label{subsub_feasibility_semi}
To make the RHC-based consensus strategy valid, it is necessary to ensure Problem \ref{problem_CRHC_Semi} is feasible at each time instant,
and the closed-loop system under the RHC-based consensus protocol can reach consensus. The feasibility is ensured in the following theorem.
\begin{theorem}\label{thm_semi_feasibility}
For the overall system (\ref{equ_overall_system}), if Problem \ref{problem_CRHC_Semi} has a solution at time instant $k$, then it admits a solution
at time instant $k+1$, for all $k\geqslant 0$.
\end{theorem}
\begin{proof}
According to the condition, we can assume that the optimal solution to Problem \ref{problem_CRHC_Semi} is $\bar{U}_{k}^*$, where $\bar{U}_{k}^* = {\rm col}(U_{k|k}^*, \cdots, U_{k+N|k}^*)$, and the corresponding optimal state sequence is $\bar{X}_{k}^* ={\rm col}(X_{k+1|k}^*, \cdots, X_{k+N|k}^*)$. At time instant $k+1$, construct a control sequence as $\bar{U}^f_{k+1} \triangleq {\rm col }(U_{k+1|k}^*,
\cdots, U_{k+N|k}^*, K^*X_{k+N|k}^*)$. The corresponding state sequence is denoted by $\bar{X}_{k+1}^f ={\rm col}(X_{k+1|k+1}^f, \cdots, X_{k+N+1|k+1}^f)$, and it is easy to see that $X^f_{k+l+1|k+1} = X^*_{k+l|k}$, $l=1, \cdots, N$. Firstly, it is true that $\bar{U}^f_{k+l|k+1}\in\mathcal{U}$, for all $l=1,\cdots, N-1$ according to the construction of $\bar{U}^f_{k+1}$.
Secondly, since $X_{k+N|k}^*\in\mathcal{O}_{\beta_s}$, it follows that $U_{k+N|k+1}^f = K^*X_{k+N|k}^* \in \mathcal{U}$, and $X^f_{k+l+1|k+1}\in\mathcal{O}_{\beta_s}$. Thus, $\bar{U}^f_{k+1}$ makes all the constraints at time $k+1$ fulfilled, and
it is a feasible solution to Problem \ref{problem_CRHC_Semi} at time $k+1$. The proof is completed.
\end{proof}

Furthermore, the consensus result for the MASs using RHC strategy is reported in the following theorem.
\begin{theorem}\label{thm_semi_RHC_consensus}
For the system in (\ref{equ_overall_system}), if the designed conditions in Theorem \ref{thm_consensus_inverse_semi} hold,
then under the designed RHC-based consensus protocol, the closed-loop system reaches {\em convergent consensus}, and the control input constraints
are fulfilled.
\end{theorem}
\begin{proof}
It is first proved that, for any state in $X_0 \in \mathcal{P}$, the system state trajectory will enter the terminal set $\mathcal{O}_{\beta_s}$,
where $\mathcal{P}$ denotes the set of all the initial states that make the input constraints and terminal constraints fulfilled.
This is proved by contradiction. Assume that the state will never enter the terminal set $\mathcal{O}_{\beta_s}$.
Define the value of the optimal cost function at time $k$ by $J_s^*(X_k, \bar{U}_k)$. According to the sub-optimality of $U_{k+1}^f$, one has
$J_s^*(X_{k+1},\bar{U}_{k+1}) - J_s^*(X_k, \bar{U}_k) \leqslant J_s(X_{k+1},U_{k+1}^f) - J_s^*(X_k, \bar{U}_k)$. Specifically, it obtains that
\begin{align*}
 & J_s(X_{k+1},U_{k+1}^f) - J_s^*(X_k, \bar{U}_k)\\
& = -\|X_k\|_{Q_s}^2 - \|U_k\|^2_{R_s}  + \|X_{k+1+N|k+1}^f\|_{S_s}^2\\
&   + \|X_{N+k|k}\|_{Q_s + (K^*)^{\rm T} R_s K^* - S_s}^2 \\
& = -\|X_k\|_{Q_s}^2 - \|U_k\|^2_{R_s} + \|X_{N+k|k}\|^2_{\Delta},
\end{align*}
where $\Delta = Q_s + (K^*)^{\rm T} R_s K^* - S_s + [(I_M\otimes A)+ (I_m\otimes B)K^*]^{\rm T}S_s[(I_M\otimes A)+ (I_m\otimes B)K^*]$.
According to the design conditions in Theorem \ref{thm_consensus_inverse_semi}, it can be seen that $\Delta = 0$. As a result, $J_s^*(X_{k+1},\bar{U}_{k+1}) - J_s^*(X_k, \bar{U}_k) \leqslant -\|X_k\|_{Q_s}^2$. Since the state trajectory will not enter the terminal set $\mathcal{O}_s$, there exists a constant $\epsilon>0$ such that $J_s^*(X_{k+1},\bar{U}_{k+1}) - J_s^*(X_k, \bar{U}_k) \leqslant - \sigma_{\min}(Q_s)\epsilon$. Making a summation from $k=0$ to $l$, one has $J_s^*(X_{l+1},\bar{U}_{l+1}) - J_s^*(X_0, \bar{U}_0) \leqslant -(l+1)\sigma_{\min}(Q_s)\epsilon$. Hence, $\lim_{l\rightarrow \infty} J_s^*(X_{l+1},\bar{U}_{l+1}) \leqslant J_s^*(X_0, \bar{U}_0)-\lim_{l\rightarrow \infty}(l+1)\sigma_{\min}(Q_s)\epsilon = -\infty$, where the fact that $J_s^*(X_0, \bar{U}_0)\geqslant 0$ and is finite is used. On the other hand, we have $J_s^*(X_{l+1},\bar{U}_{l+1})\geqslant 0$. This is a contradiction. As a result, the state trajectory will enter the
terminal set in finite steps.

Next, we prove that the closed-loop system reaches consensus by showing that it is asymptotically set-stable with respect to the set
$\mathcal{N} = {\rm Ker}(S_s)$. Assume at some time instant $k= k_1$, $X_{k_1} \in \mathcal{O}_{\beta_s}$.
Note that when $X_{k_1} \in \mathcal{O}_{\beta_s}$, all the constraints are satisfied. On the other hand, according to Theorem \ref{thm_consensus_inverse_semi}, $U_k = K^* X_k$ is optimal with respect to the performance index $J(X_k, U_k)$ with the optimal value equal to
$\|X_k\|_{S_s}^2$. As a result, $\|X_{k+N|k}\|_{S_s}^2$ can be equivalently written as $\min\{J(X_{k+N|k},\bar{U}_{k})\}$,
where $J(X_{k+N|k},\bar{U}_{k+N}) = \sum_{l=0}^{\infty}\|X_{k+N+l|k}\|_{Q_s}^2 + \|U_{k+N+l|k}\|^2_{R_s}$.
Therefore, $\min{J_s(X_k, \bar{U}_k)} = \min\{\sum_{l=0}^{N-1}\|X_{k+l|k}\|_{Q_s}^2 + \|U_{k+l|k}\|^2_{R_s}
+ \min\{\sum_{l=N}^{\infty}\|X_{k+l|k}\|_{Q_s}^2 + \|U_{k+l|k}\|^2_{R_s}\}$. According to the Dynamic Programming principle,
it can be seen that when $k\geqslant k_1$, the optimal solution to Problem \ref{problem_CRHC_Semi} is exactly $\bar{U}_k^* = {\rm col}(K^* X_{k|k}, \cdots, K^* X_{k+N-1|k})$, and control input is the optimal one $U_k = K^* X_k$.
Applying the results in Theorem \ref{thm_consensus_inverse_semi}, the closed-loop system for (\ref{equ_overall_system}) is asymptotically set stable with respect to the set $\mathcal{N} = {\rm Ker}(S_s)$, and the consensus is reached.

Finally, we prove that the closed-loop system reaches convergent consensus. It has been shown that when the state enters the terminal set, the closed-loop system becomes (\ref{equ_closed_loop_system1}). Since $\mathcal{L}$ contains a spanning tree, there exists a nonsingular matrix $T_1$ such that $\mathcal{L} = T^{-1}_1J_1 T_1$, where $J_1$ is in the Jordan form with $J_1 = {\rm diag}(0, \Lambda_1,\cdots, \Lambda_p)$. Define $T = T_1\otimes I_M$. Take a similar transform for the system in (\ref{equ_closed_loop_system1}), and denote $Y_k = T^{-1} X_k$, one has
\begin{equation}\label{equ_semi_system_transformed}
Y_{k+1} = \left[
            \begin{array}{cccc}
              A & 0 & \cdots & 0 \\
              0 & A+c\lambda_2 B K_2 & \times & \times \\
              \vdots & 0 & \ddots & \times \\
              0 & 0 & 0 & A+c\lambda_M B K_2 \\
            \end{array}
          \right] Y_k,
\end{equation}
where $\lambda_i$ are the nonzero eigenvalues of $\mathcal{L}$, $i=2,\cdots, M$.
According to the similar argument of Lemma 2 in \cite{Hengster13Auto_synchronization}
and Theorem 2 in \cite{You_11_TAC_consensus}, the necessary and sufficient condition for the system
(\ref{equ_closed_loop_system1}) to reach consensus is ${\rm spec}(A + c \lambda_i  B K_2)<1$. Since we have proved the the closed-loop system (\ref{equ_closed_loop_system1}) reaches consensus, it implies that ${\rm spec}(A + c \lambda_i  B K_2)<1$, for all $i=2,\cdots, M$ in (\ref{equ_semi_system_transformed}). Note that $A$ is semistable. As a result, the system (\ref{equ_semi_system_transformed}) is semistable.
Thus, the closed-loop system in (\ref{equ_closed_loop_system1}) is semistable,
implying that given $\|X_0\| <\infty$, $\|X_k\|$ is bounded for all $k\geqslant 0$.
Therefore, the closed-loop system will reach convergent consensus. The proof is completed.
\end{proof}

\section{Consensus for Subsystems with General Dynamics}\label{Sec_General_linear}
In this section, we extend the developed results to MASs with unstable subsystems.
Firstly, the consensus protocol that achieves optimal control performance and ensures consensus is proposed by the inverse
optimality-based approach. Then the RHC-based consensus strategy is designed. Finally, the feasibility and consensus results are presented.
\subsection{Optimal Consensus Protocol Design}
For the system in (\ref{equ_agent}), when the matrix $A$ is unstable (i.e., not semistable),
denote the $i$-th unstable eigenvalue by $\lambda_i^u(A)$,
$1\leqslant i\leqslant n$. For the unstable subsystem, we have the following result on a modified ARE.
The solution to the modified ARE depends on the properties of $A$ and $B$.
\begin{lemma}\label{lemma_ARE_like}
For the system in (\ref{equ_agent}), suppose that $(A, B)$ is controllable, and $B$ is of full column rank. Given a constant $\alpha>0$, and a symmetric matrix $Q_2>0$ such that $(A, Q_2^{\frac{1}{2}})$ is observable, then there exists a unique positive-definite matrix $S_2$ satisfying the following modified ARE:
\begin{equation}\label{equ_ARE_like}
A^{\rm T} S_2 A - S_2 + Q_2 -\frac{\delta}{1+\alpha} A^{\rm T} S_2 B (B^{\rm T} S_2 B)^{-1} B^{\rm T} S_2 A =0,
\end{equation}
if and only if $\delta>\delta_c$, where $\delta_c \triangleq \inf_{\delta}\{0\leqslant\delta\leqslant1| S_2 =A^{\rm T} S_2 A + Q_2 -\frac{\delta}{1+\alpha} A^{\rm T} S_2 B (B^{\rm T} S_2 B)^{-1} B^{\rm T} S_2 A,$ $ S_2\geqslant 0\}$.
Furthermore, $\delta_c = 1-\frac{1}{\max_i|\lambda_i^u(A)|^2}$ when $B$ is square and invertible, $\delta_c = 1-\frac{1}{\prod_i|\lambda_i^u(A)|^2}$
when $B$ is of rank one. In general, $\delta_c$ can be determined by $\delta_c = {\rm arg}\min_{\delta}\Psi_{\delta}(Y,Z)>0$ subject to $0\leqslant Y\leqslant I$, where
$\Psi_{\delta}(Y,Z) = \left[
                       \begin{array}{ccc}
                         Y & \sqrt{\delta}(Y A^{\rm T} + Z B^{\rm T}) & \sqrt{1-\delta}Y A^{\rm T} \\
                         \sqrt{\delta}(A Y + B Z^{\rm T}) & Y & 0 \\
                         \sqrt{1-\delta}A Y & 0 & Y \\
                       \end{array}
                     \right]$.
\end{lemma}
\begin{proof}
The proof can be derived by following the similar lines as in \cite{Schenato07Proceedings}\cite{Sinopoli04TAC_kalman}, so it is omitted here.
\end{proof}

Based on Lemma \ref{lemma_ARE_like}, the design condition of the consensus protocol that is optimal with respect to an optimal performance index and
guarantees consensus is reported in the following theorem.
\begin{theorem}\label{theorem_optimality_protocol}
For the system in (\ref{equ_overall_system}), suppose that $B$ is of full column rank.
If the consensus gain $K_2$ in (\ref{eqn_conventional_protocol}) is designed as $K_2 = -(B^{\rm T} S_2 B+R_2)^{-1} B^{\rm T} S_2 A$,
then the system in (\ref{equ_overall_system}) can reach consensus, and the control input $U_k = c (\mathcal{L} \otimes K_2) X_k$ is optimal with respect to the performance index $J_u(X_0,U_k) = \sum_{k=0}^{\infty} \|X_k\|_{Q_u}^2 + \|U_k\|_{R_u}^2$, with
$Q_u = S_1\otimes Q_2 + \frac{c S_1 \mathcal{L}-\delta S_1}{1+\alpha}\otimes H$, with $H = A^{\rm T} S_2 B(B^{\rm T} S_2 B)^{-1}B^{\rm T} S_2 A$,
and $R_u =  R_1\otimes R_2$, where the parameters are designed as follows:
1) $Q_2>0$ and $(A, Q_2)$ is observable; 2) $S_2$ is a symmetric and positive definite solution to (\ref{equ_ARE_like}) with a given $\delta>0$; 3) $S_1 = W \mathcal{L}$, with $W$ being symmetric and invertible, and $W \mathcal{L}$ and $W \mathcal{L}^{2}$ being symmetric; 4) $R_1 = \frac{W(I_M -c \mathcal{L})}{c \alpha}$; 5) $R_2 = \alpha B^{\rm T} S_2 B$, where $\alpha>0$ is a constant; 6) $c$ is designed such that $\frac{\delta}{\sigma_{\min}(\mathcal{L})}\leqslant c \leqslant  \frac{1}{\sigma_{\max}(\mathcal{L})}$.
\end{theorem}

\begin{proof}
The fact that $U_k = c (\mathcal{L} \otimes K_2) X_k$ is optimal with respect to the performance index $J_u(X_0,U_k)$ can be proved by following the
similar line of the proof in Theorem \ref{thm_consensus_inverse_semi}, by noticing that $S_2$ satisfies (\ref{equ_ARE_like}).

Next, it needs to be proved that $\|X_k\|_{Q_u}\geqslant \gamma_1 |X_k|_{\mathcal{N}}$,
for some constant $\gamma_1>0$ and $\mathcal{N} = {\rm Ker}(S_u)$ with $S_u = S_1 \otimes S_2$.
According to the design conditions 2) and 3), one gets that $S_1 = W \mathcal{L}$, $W>0$ and $S_2>0$.
As a result, it follows that $\mathcal{N} = {\rm Ker}(\mathcal{L}\otimes I_n)$.
On the other hand, in terms of the design condition 6) $\frac{\delta}{\sigma_{\min}(\mathcal{L})}\leqslant c$, it follows
that $c S_1 \mathcal{L}-\delta S_1\geqslant 0$. Therefore, $Q_u \geqslant 0$.
Furthermore, since $Q_2>0$ and $S_1 = W \mathcal{L}$, it can be seen that the null space of $Q_u$ is exactly $\mathcal{N}$.
As a result, one has $\|X_k\|_{Q_u} \geqslant \sigma_{\min}(Q_u)|X_k|_{\mathcal{N}}$.

Finally, by applying Lemma \ref{lemma_inverse_optimality}, the closed-loop system is asymptotically set-stable with respect to the set
$\mathcal{N}$, leading to the state consensus. The proof is completed.
\end{proof}

\begin{remark}
In comparison with the design conditions in Theorem \ref{thm_consensus_inverse_semi} for MASs with semi-stable subsystems, the design conditions 2), 3) and 6) are different for MASs with unstable subsystems.
This is due to the fact that, for the semi-stable subsystems, a Lyapunov equation in (\ref{equ_Lyapunov_equation}) can be established to design the consensus gain $K_2$, while for the unstable subsystems, only a modified ARE in (\ref{equ_ARE_like}) can be found to design the consensus gain.
This difference also results in a different optimal performance index.
\end{remark}

\begin{remark}
By comparing the design condition 6) in Theorems \ref{theorem_optimality_protocol} and \ref{thm_consensus_inverse_semi}, it is noted that the design condition for the coupling factor $c$ for the MASs with unstable subsystems is more stringent than that of semi-stable cases. In fact, for the MASs with unstable subsystems, in order to make such a $c$ exists, one requires that $\delta\leqslant \frac{\sigma_{\min}(\mathcal{L})}{\sigma_{\max}(\mathcal{L})}$. But according to Lemma \ref{lemma_ARE_like}, $\delta>\delta_c$ is a parameter determined by the system matrices $A$ and/or $B$, and $\frac{\sigma_{\min}(\mathcal{L})}{\sigma_{\max}(\mathcal{L})}$ is fixed parameter for the connected networks. As a result, there may exist unstable subsystems such that $\delta_c > \frac{\sigma_{\min}(\mathcal{L})}{\sigma_{\max}(\mathcal{L})}$. For such subsystems, there might not exist an optimal consensus protocol.
However, for the MASs with semi-stable subsystems, the coupling factor $c$ can always be chosen to satisfy the condition 6) in Theorem \ref{thm_consensus_inverse_semi}.
\end{remark}

\subsection{RHC-Based Consensus Strategy}
The design of the terminal set is similar as that of semi-stable cases, i.e., there exists a $\beta_u>0$ such that
$\mathcal{O}_{\beta_u}$ is forward invariant for the system in (\ref{equ_closed_loop_system1}), and $K^* X_k \in \mathcal{U}$, for all $X_k \in \mathcal{O}_{\beta_u}$.
Hence, the core of the RHC-based consensus strategy is to solve the following constrained optimization problem:
\begin{problem}\label{problem_CRHC}
$\bar{U}_k^* = {\rm arg} \min J_u(X_k,\bar{U}_k), {\text subject \; to}$
\begin{align*}
X_{k+l+1|k} = &(I_M \otimes A) X_{k+l|k} + (I_M \otimes B) U_{k+l|k},\\
& U_{k+l|k}\in\mathcal{U}, X_{k+N|k}\in\mathcal{O}_{\beta_u},
\end{align*}

where $l=0,\cdots, N-1$, $X_{k|k} = X_k$, and $\bar{U}_k = {\rm col}(U_{k|k},\cdots, U_{k+N-1|k})$.
\end{problem}
The cost function is defined as
\begin{align*}
J_u(X_k,\bar{U}_k) = \sum_{i=0}^{N-1}\|X_{k+i|k}\|_{Q_u}^2 + \|U_{k+i|k}\|^2_{R_u} + \|X_{k+N|k}\|^2_{S_u},
\end{align*}
where $Q_u$, $R_u$ and $S_u$ the parameters in Theorem \ref{theorem_optimality_protocol}, respectively.

Based on Problem \ref{problem_CRHC}, the centralized RHC-based consensus strategy is: At each time instant $k$, Problem \ref{problem_CRHC} is
solved to generate $\bar{U}_k^*$, and the control input $U_k$ is taken as $U_{k+1|k}^*$.
Analogously, the cost function can be distributively assigned to each agent $i$ under certain condition.
\begin{lemma}\label{lemma_distribution}
In the cost function $J_u(X_k,\bar{U}_k)$, if $W = \mu I_M$, and $\mathcal{L} = \mathcal{L}^{\rm T}$, where the scalar $\mu >0$, then
$J_u(X_k,\bar{U}_k)$ can be distributively assigned to each agent $i$ by the following sub-cost function as
\begin{align*}
J_u^i(x^i_k,\bar{u}^i_k) =& \sum_{l=0}^{N-1} \sum_{j\in \mathcal{N}_i}a_{ij}\mu[(x_{k+l|k}^i)^{\rm T} (Q_2-\frac{H}{1+\alpha}) (x_{k+l|k}^i\\
                          & -x_{k+l|k}^j)] + \frac{\mu}{c\alpha} \|u_{k+l|k}^i\|_{R_2}^2\\
                          & - \frac{\mu}{\alpha} [(u_{k+l|k}^i)^{\rm T} R_2 (u_{k+l|k}^i-u_{k+l|k}^j)]\\
                          & + \frac{c \mu }{1+\alpha}\|\sum_{j\in \mathcal{N}_i}a_{ij}(x_{k+l|k}^i-x_{k+l|k}^j)\|_{H}^2 \\
                          & + \mu \sum_{j\in \mathcal{N}_i}a_{ij}[(x_{k+N|k}^i)^{\rm T} S_2 (x_{k+N|k}^i-x_{k+N|k}^j)].
\end{align*}
\end{lemma}
\begin{proof}
The proof can be obtained by following the similar line as that of Lemma \ref{lemma_distribution_Semi}, so it is omitted here.
\end{proof}

Likely, the terminal constraint can be equivalently imposed to each agent $i$ as $\sum_{j\in \mathcal{N}_i}a_{ij}\mu(x_{k+N|k}^i)^{\rm T}$
$ S_2 (x_{k+N|k}^i-x_{k+N|k}^j) \leqslant \frac{\beta_u^2}{M}$. And the optimization problem associated with each agent $i$ can be formulated as
\begin{problem}\label{problem_distribution}
$\bar{u}_k^{i*} = {\rm arg} \min J_u^i(x_k^i,\bar{u}_k^i), {\text subject \; to}$
\begin{align*}
& x_{k+l+1|k}^i =A x_{k+l|k}^i + B u_{k+l|k}^i,\\
& u_{k+l|k}^i \in\mathcal{U}^i, l=0\cdots, N-1, \\
& \sum_{j\in \mathcal{N}_i}a_{ij}\mu(x_{k+N|k}^i)^{\rm T} S_2 (x_{k+N|k}^i-x_{k+N|k}^j) \leqslant \frac{\beta_u^2}{M},
\end{align*}
where $\bar{u}_{k}^i = {\rm col}(u_{k|k}^i, \cdots, u_{k+N-1|k}^i)$, and $x_k^i = x_{k|k}^i$.
\end{problem}

The distributed RHC-based consensus strategy is the same as that for MASs with semi-stable subsystems by replacing the solution to Problem \ref{problem_semi_distribution} with that to Problem \ref{problem_distribution}.

\subsection{Feasibility and Consensus Property}
The feasibility result for the MASs of unstable subsystems is similar as that for semi-stable case, which is presented in the
following corollary.
\begin{corollary}\label{corollary_feasibility}
For the system in (\ref{equ_overall_system}), if Problem \ref{problem_CRHC} has a solution at time instant $k$, then it has a solution
at time instant $k+1$, for all $k\geqslant 0$.
\end{corollary}

Due to the unstable modes of the subsystems, the closed-loop system under the designed consensus protocol can reach consensus,
rather than convergent consensus as in the case of semi-stable subsystems.
\begin{theorem}\label{thm_RHC_consensus}
For the system in (\ref{equ_overall_system}), suppose that the designed conditions in Theorem \ref{theorem_optimality_protocol} hold.
Then under the designed RHC-based consensus protocol, the control input constraints are fulfilled and the closed-loop system reaches consensus.
\end{theorem}
\begin{proof}
The proof can be obtained by using the first and second part as the proof of Theorem \ref{thm_semi_RHC_consensus}.
\end{proof}

\section{Discussions on Design Conditions}\label{Sec_discussions}
In this section, discussions and insights are provided for the parameter design in the optimal consensus protocols (i.e., Theorem \ref{thm_consensus_inverse_semi} and \ref{theorem_optimality_protocol}) and RHC-based consensus strategies.
\subsection{Constraints for Cost Functions}
{\em Couplings in cost functions}: Unlike the conventional LQR problem and RHC strategy, the design of parameters in the cost functions, i.e.,
$Q_s, Q_u$, $R_s, R_u$ and $S_s, S_u$ has more constraints. In particular, the cost functions
are coupled with the network topologies (i.e., $S_1 = W \mathcal{L}$, and $R_1 = \frac{W(I_M -c \mathcal{L})}{c \alpha}$) and system matrices (i.e., $H=A^{\rm T} S_2 B(B^{\rm T} S_2 B)^{-1}B^{\rm T} S_2 A$ and $R_2 = \alpha B^{\rm T}S_2 B$).
This is due to the fact that the optimality and consensus are required simultaneously.

{\em Conditions for $Q_2$ and $S_2$}: The design constraints for $Q_2$ and $S_2$ come from the unstable eigenvalues of system matrix $A$.
For semi-stable subsystems, $Q_2$ needs to satisfy the condition that makes $S_2>0$ be a solution to the Lyapunov like equation (\ref{equ_Lyapunov_equation}), while for the unstable subsystems, $Q_2$ needs to satisfy more stringent condition that renders $S_2$ to be a solution to the modified ARE in (\ref{equ_ARE_like}). In fact, if the subsystem is stable, it only is required that $Q_2>0$ and $S_2>0$.

{\em Conditions for $c$}: For semi-stable subsystems, $c$ is required to satisfy $0< c\leqslant\frac{1}{\sigma_{\max}(\mathcal{L})}$, ensuring $R_1\geqslant 0$. But for unstable subsystems, $c$ needs to satisfy $\frac{\delta}{\sigma_{\min}(\mathcal{L})}\leqslant c\leqslant \frac{1}{\sigma_{\max}(\mathcal{L})}$, which is a coupled constraint from the system matrices and the network topology, and is more stringent.

In conclusion, the unstable modes of subsystems strengthen design constraints for $Q_2$, $S_2$ and $c$.

\subsection{Constraints for Network Topology}
The constraints in network topology are imposed by $W\mathcal{L}$ and $W\mathcal{L}^2$ being symmetric, where $W>0$ is symmetric (i.e., condition 3) in Theorem \ref{thm_consensus_inverse_semi} and \ref{theorem_optimality_protocol}).
This constraint arises from the symmetry requirement of the cost function in optimal control. The following lemma simplifies the constraint in $\mathcal{L}$.
\begin{lemma}\label{lemma_symmetric}
If $W\mathcal{L}$ and $W>0$ are symmetric, then $W\mathcal{L}^2$ is also symmetric.
\end{lemma}
\begin{proof}
$(W\mathcal{L}^2)^{\rm T} = \mathcal{L}^{\rm T}\mathcal{L}^{\rm T} W^{\rm T} = \mathcal{L}^{\rm T}  W \mathcal{L} =  W \mathcal{L} \mathcal{L} =
W \mathcal{L}^2$, where the fact that $W \mathcal{L}$ and $W$ are symmetric is used.
\end{proof}

The network graphs that satisfy the condition that $W\mathcal{L}$ and $W>0$ are symmetric can be found in the following classes \cite{Hengster14TAC_consensus_optimal}: 1) undirected graphs, 2) detailed balanced graphs and 3) diagraphs with simple Laplacian.

\section{Simulation Studies}\label{Sec_simulation}
In this section, two examples on MASs with semi-stable and unstable subsystems are given to verify the proposed theoretical results.
\subsection{Example: Semi-stable Case}
Consider an MAS with semi-stable subsystems studied in \cite{Hui09IJC_semistable}. By discretizing it with the period $T=0.1 s$, the system parameters are as follows:
$A = \left[
       \begin{array}{ccccc}
         0.8 & 0.1 & 0.1 & 0 & 0 \\
         0 & 0.9 & 0 & 0.1 & 0 \\
         0.1 & 0.1 & 0.6 & 0.1 & 0.1 \\
         0 & 0.1 & 0.1 & 0.8 & 0 \\
         0.1 & 0.1 & 0 & 0 & 0.8 \\
       \end{array}
     \right]
$, $B =\left[
         \begin{array}{cc}
           -0.1 & 0.1 \\
           0.1 & -0.2 \\
           0 & -0.3 \\
           0.08 & 0.1 \\
           0.2 & 0.08 \\
         \end{array}
       \right]$.
The control input for each agent $i$ is required to satisfy  the constraints: $-0.3 \leqslant u_k^i(1)\leqslant 0.3$, and $-0.3 \leqslant u_k^i(2)\leqslant 0.3$.
The MAS under consideration consists of $5$ agents. The communication network contains a spanning tree, and its Laplacian matrix is figured out as
$\mathcal{L} = \left[
                 \begin{array}{ccccc}
                   2 & -1 & 0 & -1 & 0 \\
                   -1 & 2 & -1 & 0 & 0 \\
                   0 & -1 & 2 & 0 & -1 \\
                   -1 & 0 & 0 & 2 & -1 \\
                   0 & 0 & -1 & -1 & 2 \\
                 \end{array}
               \right]$.

The parameters are designed as follows:
$Q_2 = \left[
                                                 \begin{array}{ccccc}
                                                   1 & 0 & 0 & 0 & -1 \\
                                                   0 & 1 & 0 & 0 & -1 \\
                                                   0 & 0 & 1 & 0 & -1 \\
                                                   0 & 0 & 0 & 1 & -1 \\
                                                   -1 & -1 & -1 & -1 & 4 \\
                                                 \end{array}
                                               \right]
$, \\
$S_2 = \left[
            \begin{array}{ccccc}
2.551 &	-0.447 &	0.119 &	-0.813 &	-1.069\\
-0.447 &	4.028 &	0.227 &	1.356 &	-2.664\\
0.119& 	0.227&	1.799 &	0.740 &	-2.431\\
-0.813&	1.356 &	0.740 &	3.884 &	-3.689 \\
-1.069	& -2.664	& -2.431 & -3.689 &	10.081\\
            \end{array}
          \right]$,
$\alpha = 10$, $c = 10$ and $W = 0.5 I_5$. Note these designed parameters satisfies all the design conditions in Theorem \ref{thm_consensus_inverse_semi}.
In the RHC-based consensus strategy, the prediction horizon is taken as $N = 9$. By utilizing the MATLAB package, the simulation results are reported
in the following figures.

\begin{figure}[!hbt] \centering
\includegraphics[width=0.5\textwidth]{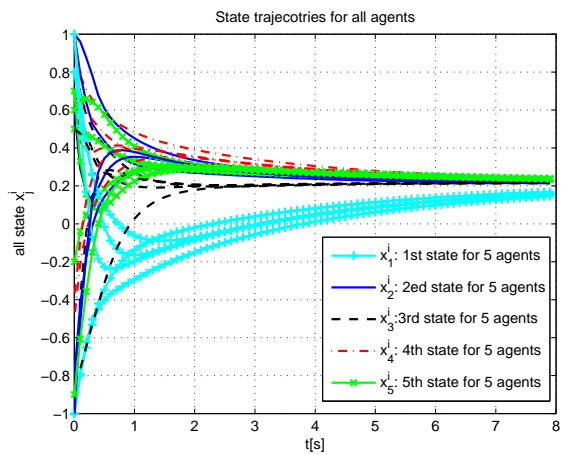}
\caption{State trajectories for all the 5 agents.}\label{fig_state_semi}
\end{figure}

\begin{figure}[!hbt] \centering
\includegraphics[width=0.5\textwidth]{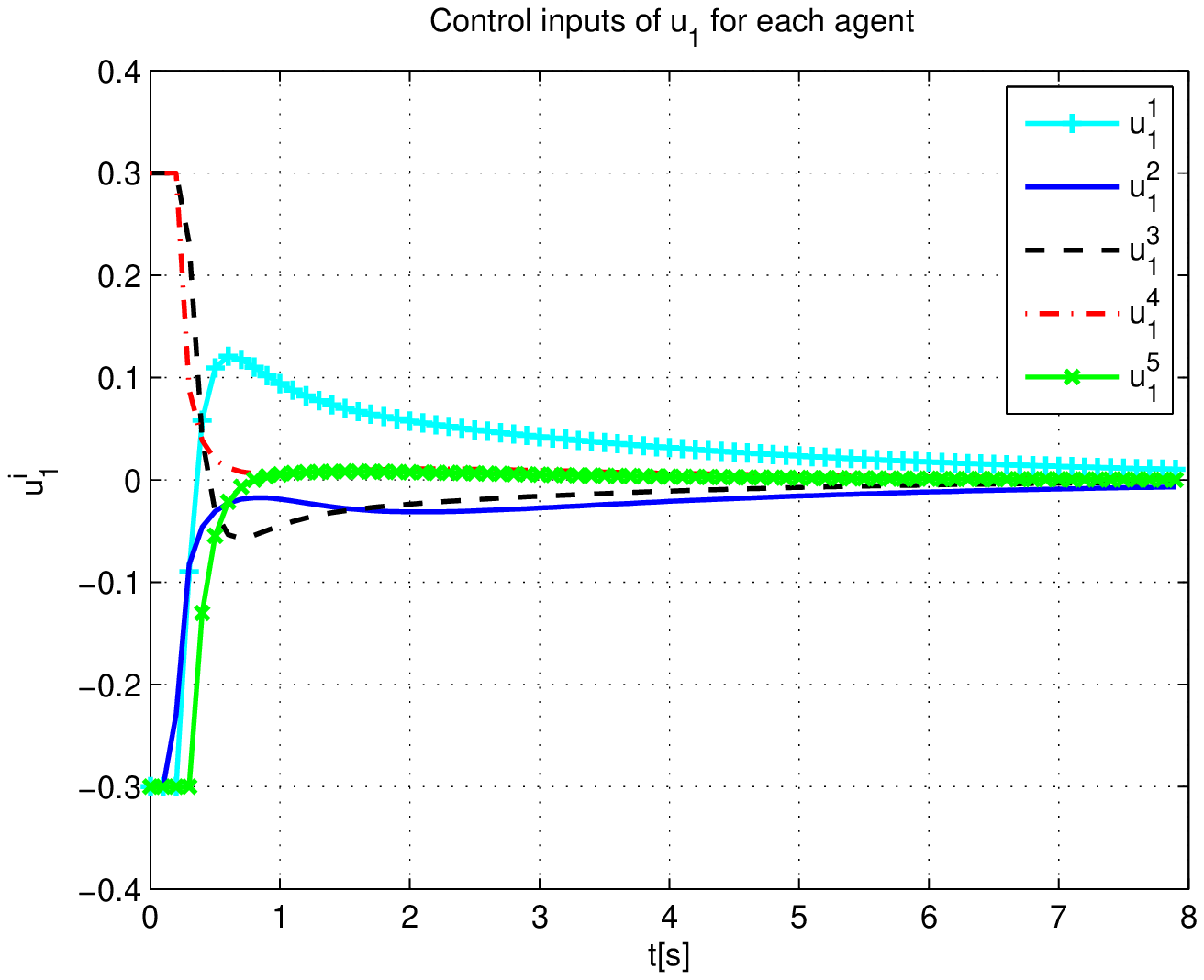}
\caption{The first control inputs for all the 5 agents.}\label{fig_control1_semi}
\end{figure}

\begin{figure}[!h] \centering
\includegraphics[width=0.5\textwidth]{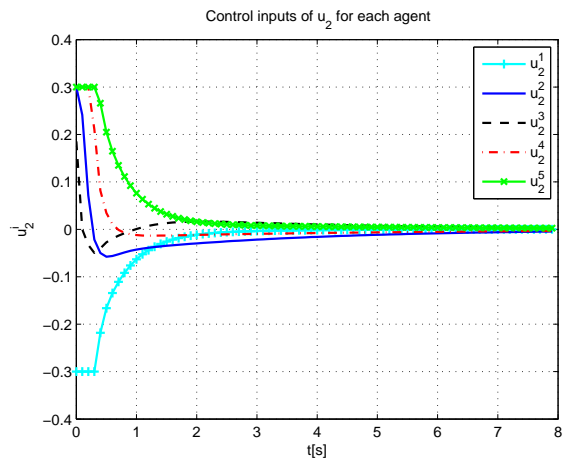}
\caption{The second control inputs for all the 5 agents.}\label{fig_control2_semi}
\end{figure}

From Fig. \ref{fig_state_semi}, it can be seen that the first states for all the 5 agents
converge to the same value, and this is also true for the other states.
This implies that the closed-loop system reaches consensus under the designed RHC-based consensus strategy.
In addition, all the states for the 5 agents is finally bounded, verifying that the closed-loop system reaches
convergent consensus as proved in Theorem \ref{thm_semi_RHC_consensus}.
The control inputs for the 5 agents are shown in Figs. \ref{fig_control1_semi} and \ref{fig_control2_semi}.
it can be seen that the control input constraints are satisfied, indicating that the proposed RHC-based consensus protocol
can meet the pre-scribed control input constraints.

\subsection{Example: Unstable Case}
Consider an MAS with $5$ agents with the dynamics being unstable \cite{Hengster13Auto_synchronization}.
The system matrices for each subsystem are as follows:
$A = \left[
       \begin{array}{ccc}
         0 & 1 & 0 \\
         0 & 0 & 1 \\
         -0.2 & 0.2 & 1.1 \\
       \end{array}
     \right]
$, and $B = \left[
          \begin{array}{c}
            0 \\
            0 \\
            1 \\
          \end{array}
        \right]$.
Note that $A$ is unstable and $(A, B)$ is controllable.
The control input for each agent is required to satisfy the constraint $-1\leqslant u^i\leqslant 1$ for all $i$.
The communication network contains a spanning tree, and its Laplacian matrix is obtained as follows:
$\mathcal{L} = \left[
                 \begin{array}{ccccc}
                   4 & -1 & -1 & -1 & -1 \\
                   -1 & 4 & -1 & -1 & -1 \\
                   -1 & -1 & 4 &-1 & -1 \\
                   -1 & -1 & -1 & 4 & -1 \\
                  -1 & -1 & -1 & -1 & 4\\
                 \end{array}
               \right]$.
The parameters are designed as follows: $Q_2 = \left[
\begin{array}{ccc}
3.990 & 1.027 & 3.069\\
1.027 &	2.833 &	1.426\\
3.069 &	1.426&	3.949\\
\end{array}
\right]$, $\delta = 0.1634$. According to modified ARE in (\ref{equ_ARE_like}),
$S_2 = \left[
 \begin{array}{ccc}
 4 & 1 & 3 \\
 1 & 6 & 2 \\
 3 & 2 & 10 \\
 \end{array}
 \right]$. $\alpha = 0.0274$, $W = 0.5 I_5$ and $c = 0.2$.
Note that $\frac{\delta}{\sigma_{\min}(\mathcal{L})} = 0.0327$ and
$\frac{1}{\sigma_{\max}(\mathcal{L})} = 0.2$.
Therefore, all the design conditions in Theorem \ref{theorem_optimality_protocol} are satisfied.
Under the designed RHC-based consensus protocol, we use the MATLAB software to conduct the simulation again.
The simulation results are displayed in Figs. \ref{fig_state1} - \ref{fig_input}.
From Figs.\ref{fig_state1} to \ref{fig_state3}, it can been observed that the closed-loop system
reaches state consensus, but the consensus point is divergent, differing from that for MASs with semi-stable subsystems.
The control input is shown in Fig. \ref{fig_input}, which implies that the prescribed control input constraints are fulfilled.
As a consequence, the theoretical results for MASs with unstable subsystems are verified.
\begin{figure}[!hbt] \centering
\includegraphics[width=0.5\textwidth]{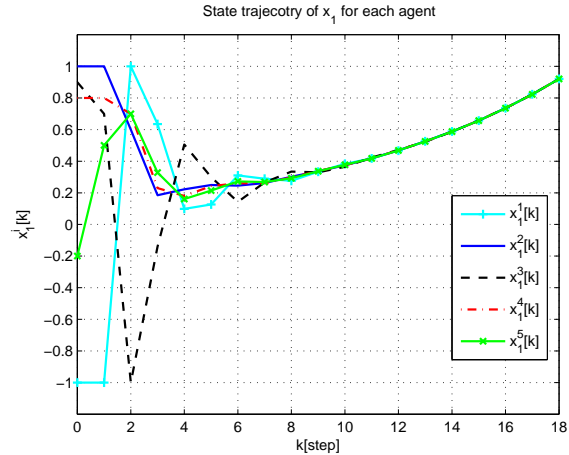}
\caption{The first states for the 5 agents.}\label{fig_state1}
\end{figure}

\begin{figure}[!hbt] \centering
\includegraphics[width=0.5\textwidth]{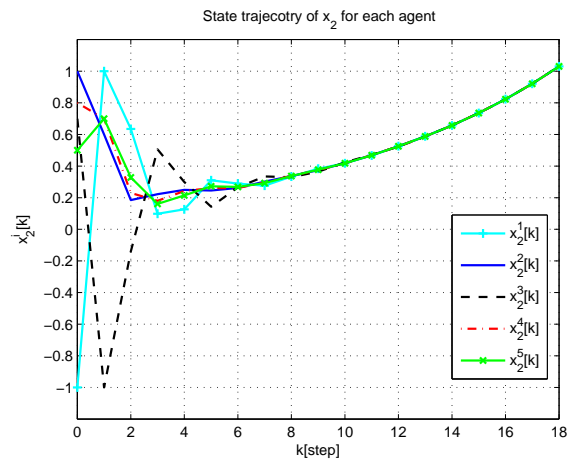}
\caption{The second states for the 5 agents.}\label{fig_state2}
\end{figure}

\begin{figure}[!h] \centering
\includegraphics[width=0.5\textwidth]{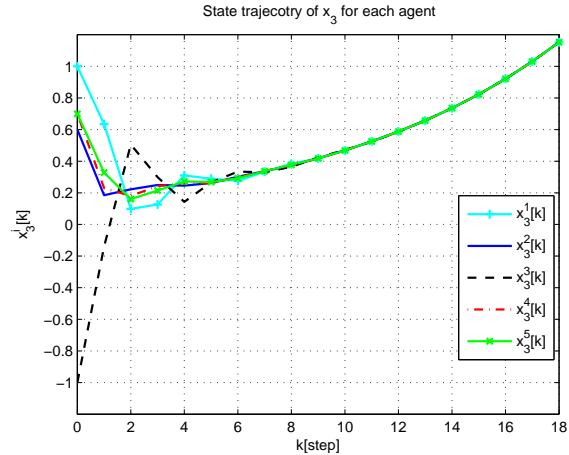}
\caption{The third states for the 5 agents.}\label{fig_state3}
\end{figure}

\begin{figure}[!h] \centering
\includegraphics[width=0.5\textwidth]{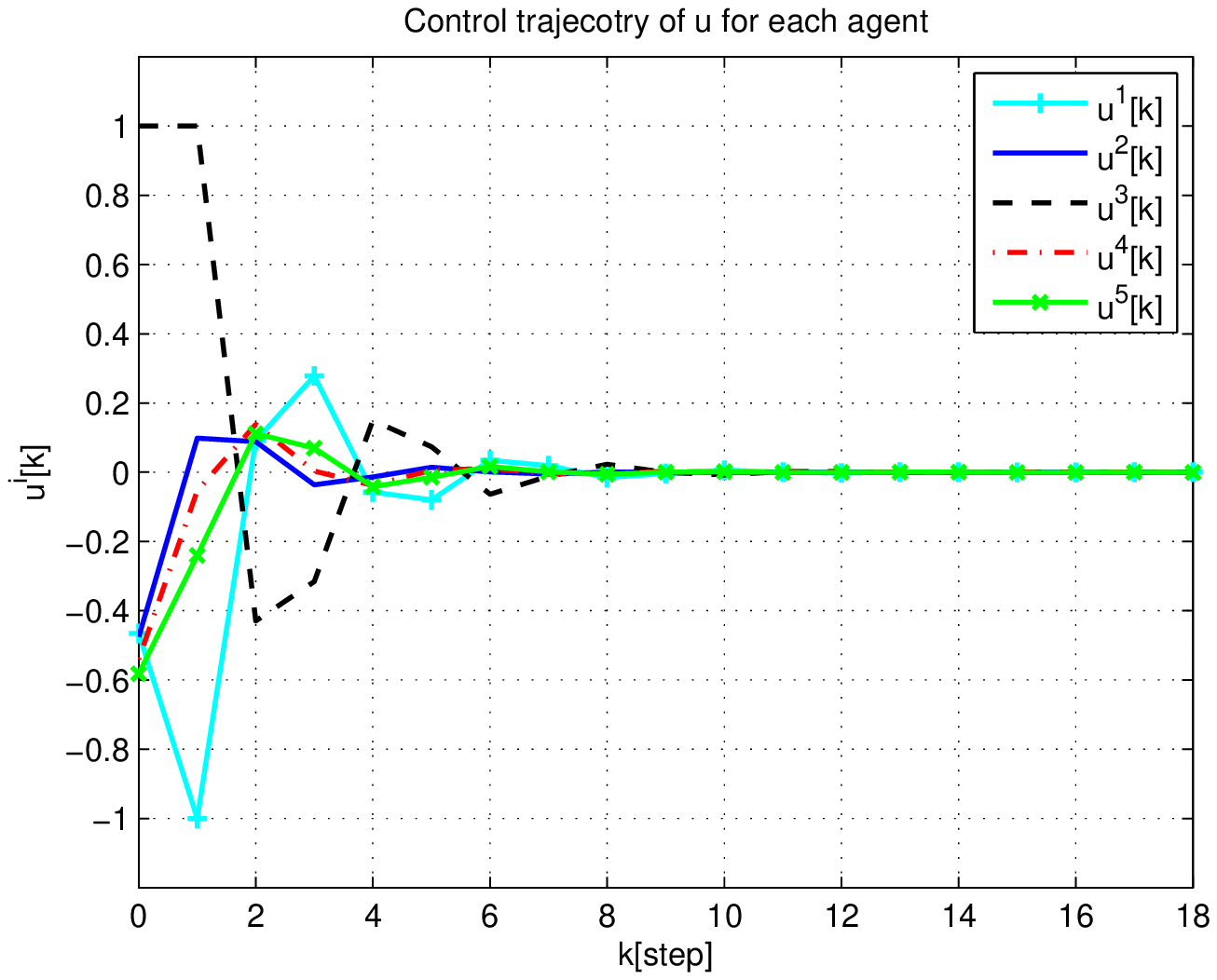}
\caption{The control inputs for the 5 agents.}\label{fig_input}
\end{figure}

\section{CONCLUSIONS}\label{Sec_conclusion}
In this paper, we have studied the RHC-based consensus problem for input-constrained MASs with semi-stable and unstable subsystems.
The results on designing the optimal consensus protocols have been firstly proposed for such two classes of MASs without constraints, respectively. Based on the designed optimal consensus protocols, the RHC-based consensus strategies have been designed.
Furthermore, the feasibility of the designed RHC-based consensus strategies
and the consensus properties of the closed-loop MASs have been analyzed.
It is shown that the achieved global optimal performance indices by the optimal consensus protocol are coupled with the system matrices of each subsystems and the network topology. The designed consensus strategies can make the input constraints fulfilled and the closed-loop system reach consensus. In particular, for the MASs with semi-stable subsystems, the closed-loop system can reach convergent consensus.

%
%
%

\bibliographystyle{IEEEtran}
\bibliography{ref}

\begin{thebibliography}{10}
\providecommand{\url}[1]{#1}
\csname url@samestyle\endcsname
\providecommand{\newblock}{\relax}
\providecommand{\bibinfo}[2]{#2}
\providecommand{\BIBentrySTDinterwordspacing}{\spaceskip=0pt\relax}
\providecommand{\BIBentryALTinterwordstretchfactor}{4}
\providecommand{\BIBentryALTinterwordspacing}{\spaceskip=\fontdimen2\font plus
\BIBentryALTinterwordstretchfactor\fontdimen3\font minus
  \fontdimen4\font\relax}
\providecommand{\BIBforeignlanguage}[2]{{%
\expandafter\ifx\csname l@#1\endcsname\relax
\typeout{** WARNING: IEEEtran.bst: No hyphenation pattern has been}%
\typeout{** loaded for the language `#1'. Using the pattern for}%
\typeout{** the default language instead.}%
\else
\language=\csname l@#1\endcsname
\fi
#2}}
\providecommand{\BIBdecl}{\relax}
\BIBdecl

\bibitem{Reza_04_consensus_TAC}
R.~Olfati-Saber and R.~M. Murray, ``{Consensus problems in networks of agents
  with switching topology and time-delays},'' \emph{IEEE Transactions on
  Automatic Control}, vol.~49, pp. 1520--1533, 2004.

\bibitem{Moreau_05_TAC_consensus}
L.~Moreau, ``{Stability of multiagent systems with time-dependent communication
  links},'' \emph{IEEE Transactions on Automatic Control}, vol.~50, pp.
  169--182, 2005.

\bibitem{Wei_05_TAC_consensus}
W.~Ren and R.~W. Beard, ``{Consensus seeking in multiagent systems under
  dynamically changing interaction topologies},'' \emph{IEEE Transactions on
  Automatic Control}, vol.~50, pp. 655--661, 2005.

\bibitem{Nedic10TAC_constrained_consensus}
A.~Nedic, A.~Ozdaglar, and P.~A. Parrilo, ``Constrained consensus and
  optimization in multi-agent networks,'' \emph{IEEE Transactions on Automatic
  Control}, vol.~55, no.~4, pp. 922--938, 2010.

\bibitem{Lin14TAC_constrained_consensus}
P.~Lin and W.~Ren, ``Constrained consensus in unbalanced networks with
  communication delays,'' \emph{IEEE Transactions on Automatic Control,},
  vol.~59, no.~3, pp. 775--781, 2014.

\bibitem{Wang14SCL_synchronization}
Q.~Wang, C.~Yu, and H.~Gao, ``Synchronization of identical linear dynamic
  systems subject to input saturation,'' \emph{Systems \& Control Letters},
  vol.~64, pp. 107--113, 2014.

\bibitem{Borrelli08TAC_LQR}
F.~Borrelli and T.~Keviczky, ``Distributed lqr design for identical dynamically
  decoupled systems,'' \emph{IEEE Transactions on Automatic Control}, vol.~53,
  no.~8, pp. 1901--1912, 2008.

\bibitem{Johansson08Auto_consensus}
B.~Johansson, A.~Speranzon, M.~Johansson, and K.~H. Johansson, ``On
  decentralized negotiation of optimal consensus,'' \emph{Automatica}, vol.~44,
  no.~4, pp. 1175--1179, 2008.

\bibitem{Cao10TCYB_optimal_consensus}
Y.~Cao and W.~Ren, ``Optimal linear-consensus algorithms: an lqr perspective,''
  \emph{IEEE Transactions on Systems, Man, and Cybernetics, Part B:
  Cybernetics}, vol.~40, no.~3, pp. 819--830, 2010.

\bibitem{Hengster14TAC_consensus_optimal}
K.~Hengster-Movric and F.~Lewis, ``Cooperative optimal control for multi-agent
  systems on directed graph topologies,'' \emph{IEEE Transactions on Automatic
  Control}, vol.~59, no.~3, pp. 769--774, 2014.

\bibitem{Dunbar_06_Auto_DMPC}
W.~B. Dunbar and R.~M. Murray, ``Distributed receding horizon control for
  multi-vehicle formation stabilization,'' \emph{Automatica}, vol.~42, no.~4,
  pp. 549--558, 2006.

\bibitem{Li_Auto_14_DRHC}
H.~Li and Y.~Shi, ``Distributed receding horizon control of large-scale
  nonlinear systems: Handling communication delays and disturbances,''
  \emph{Automatica}, vol.~50, no.~4, pp. 1264--1271, 2014.

\bibitem{Franco_08_TAC_DMPC}
E.~Franco, L.~Magni, T.~Parisini, M.~M. Polycarpou, and D.~M. Raimondo,
  ``Cooperative constrained control of distributed agents with nonlinear
  dynamics and delayed information exchange: {A stabilizing receding-horizon
  approach},'' \emph{IEEE Transactions on Automatic Control}, vol.~53, no.~1,
  pp. 324--338, 2008.

\bibitem{Richards_IJC_07_RDMPC}
A.~Richards and J.~P. How, ``Robust distributed model predictive control,''
  \emph{International Journal of Control}, vol.~80, no.~9, pp. 1517--1531,
  2007.

\bibitem{Li_TAC_14_RDRHC}
H.~Li and Y.~Shi, ``Robust distributed model predictive control of constrained
  continuous-time nonlinear systems: A robustness constraint approach,''
  \emph{IEEE Transactions on Automatic Control}, vol.~59, no.~6, pp.
  1673--1678, 2014.

\bibitem{Muller_12_DMPC_IJRNC}
M.~A. M\"{u}ller, M.~Reble, and F.~Allg\"{o}wer, ``Cooperative control of
  dynamically decoupled systems via distributed model predictive control,''
  \emph{International Journal of Robust and Nonlinear Control}, vol.~22,
  no.~12, pp. 1376--1397, 2012.

\bibitem{Ferrari_09_TAC_MPC_consensus}
G.~Ferrari-Trecate, L.~Galbusera, M.~P.~E. Marciandi, and R.~Scattolini,
  ``Model predictive control schemes for consensus in multi-agent systems with
  single- and double-integrator dynamics,'' \emph{IEEE Transactions on
  Automatic Control}, vol.~54, no.~11, pp. 2560--2572, 2009.

\bibitem{Zhan_13_auto_consensus}
J.~Zhan and X.~Li, ``Consensus of sampled-data multi-agent networking systems
  via model predictive control,'' \emph{Automatica}, vol.~49, no.~8, pp. 2502
  -- 2507, 2013.

\bibitem{Zhang15TCSI}
H.-T. Zhang, Z.~Cheng, and G.~Chen, ``Model predictive flocking control for
  second-order multi-agent systems with input constraints,'' \emph{IEEE
  Transactions on Circuits and Systems I-Regular paper}, vol.~62, no.~6, pp.
  1599--1606, 2015.

\bibitem{Huiping14Auto_Consensus}
H.~Li and W.~Yan, ``Receding horizon control based consensus scheme in general
  linear multi-agent systems,'' \emph{Automatica}, vol.~49, no.~7, pp.
  1031--1036, 2015.

\bibitem{Jiang02SCL_inverse_set_stability}
Z.-P. Jiang and Y.~Wang, ``A converse lyapunov theorem for discrete-time
  systems with disturbances,'' \emph{Systems \& Control Letters}, vol.~45, pp.
  49--58, 2002.

\bibitem{Ma10TAC_IFF_consensus}
C.-Q. Ma and J.-F. Zhang, ``Necessary and sufficient conditions for
  consensusability of linear multi-agent systems,'' \emph{IEEE Transactions on
  Automatic Control}, vol.~55, no.~5, pp. 1263--1268, 2010.

\bibitem{Frank1986}
F.~L. Lewis, \emph{Optimal Control}.\hskip 1em plus 0.5em minus 0.4em\relax
  John Wiley \& Sons, 1986.

\bibitem{Hui09IJC_semistable}
Q.~Hui and W.~M. Haddad, ``Optimal semistable stabilization for linear
  discrete-time dynamical systems with applications to network consensus,''
  \emph{International Journal of Control}, vol.~82, no.~3, pp. 456--469, 2009.

\bibitem{You_11_TAC_consensus}
Y.~K. and X.~L., ``Network topology and communication data rate for
  consensusability of discrete-time multi-agent systems,'' \emph{IEEE
  Transactions on Automatic Control}, vol.~56, no.~10, pp. 2262--2275, 2011.

\bibitem{Hengster13Auto_synchronization}
K.~Hengster-Movric, K.~You, F.~L. Lewis, and L.~Xie, ``Synchronization of
  discrete-time multi-agent systems on graphs using riccati design,''
  \emph{Automatica}, vol.~49, no.~2, pp. 414--423, 2013.

\bibitem{Schenato07Proceedings}
L.~Schenato, B.~Sinopoli, M.~Franceschetti, K.~Poolla, and S.~S. Sastry,
  ``Foundations of control and estimation over lossy networks,''
  \emph{Proceedings of the IEEE}, vol.~95, no.~1, pp. 163--187, 2007.

\bibitem{Sinopoli04TAC_kalman}
B.~Sinopoli, L.~Schenato, M.~Franceschetti, K.~Poolla, M.~I. Jordan, and S.~S.
  Sastry, ``Kalman filtering with intermittent observations,'' \emph{IEEE
  Transactions on Automatic Control}, vol.~49, no.~9, pp. 1453--1464, 2004.

\end{thebibliography}

\end{document}